\documentclass[12pt]{article}
\usepackage{amssymb, amsmath, amsthm, amscd}
\usepackage[utf8]{inputenc}
\usepackage[all,cmtip]{xy}
\usepackage{enumitem}
\usepackage{setspace}
\usepackage{mathdots}
\usepackage{mathtools}

\makeatletter
\newcommand{\xtwoheadrightarrow}[2][]{\ext@arrow 0579\twoheadrightarrowfill@{#1}{#2}}
\newcommand{\twoheadrightarrowfill@}{\arrowfill@\relbar\relbar\twoheadrightarrow}
\makeatother

\usepackage[mathscr]{euscript}
\usepackage[colorlinks=true]{hyperref}
\hypersetup{colorlinks   = true}
\hypersetup{linkcolor=blue}
\usepackage{tikz}

\begin{document}

\mathchardef\mhyphen="2D
\newtheorem{The}{Theorem}[section]
\newtheorem{Lem}[The]{Lemma}
\newtheorem{Prop}[The]{Proposition}
\newtheorem{Cor}[The]{Corollary}
\newtheorem{Rem}[The]{Remark}
\newtheorem{Obs}[The]{Observation}
\newtheorem{SConj}[The]{Standard Conjecture}
\newtheorem{Titre}[The]{\!\!\!\! }
\newtheorem{Conj}[The]{Conjecture}
\newtheorem{Question}[The]{Question}
\newtheorem{Prob}[The]{Problem}
\newtheorem{Def}[The]{Definition}
\newtheorem{Not}[The]{Notation}
\newtheorem{Claim}[The]{Claim}
\newtheorem{Conc}[The]{Conclusion}
\newtheorem{Ex}[The]{Example}
\newtheorem{Fact}[The]{Fact}
\newtheorem{Formula}[The]{Formula}
\newtheorem{Formulae}[The]{Formulae}
\newtheorem{The-Def}[The]{Theorem and Definition}
\newtheorem{Prop-Def}[The]{Proposition and Definition}
\newtheorem{Lem-Def}[The]{Lemma and Definition}
\newtheorem{Cor-Def}[The]{Corollary and Definition}
\newtheorem{Conc-Def}[The]{Conclusion and Definition}
\newtheorem{Terminology}[The]{Note on terminology}
\newcommand{\C}{\mathbb{C}}
\newcommand{\R}{\mathbb{R}}
\newcommand{\N}{\mathbb{N}}
\newcommand{\Z}{\mathbb{Z}}
\newcommand{\Q}{\mathbb{Q}}
\newcommand{\Proj}{\mathbb{P}}
\newcommand{\Rc}{\mathcal{R}}
\newcommand{\Oc}{\mathcal{O}}
\newcommand{\Vc}{\mathcal{V}}
\newcommand{\Id}{\operatorname{Id}}
\newcommand{\pr}{\operatorname{pr}}
\newcommand{\rk}{\operatorname{rk}}
\newcommand{\del}{\partial}
\newcommand{\delbar}{\bar{\partial}}
\newcommand{\Cdot}{{\raisebox{-0.7ex}[0pt][0pt]{\scalebox{2.0}{$\cdot$}}}}
\newcommand\nilm{\Gamma\backslash G}
\newcommand\frg{{\mathfrak g}}
\newcommand{\fg}{\mathfrak g}
\newcommand{\Oh}{\mathcal{O}}
\newcommand{\Kur}{\operatorname{Kur}}
\newcommand\gc{\frg_\mathbb{C}}
\newcommand\dan[1]{{\textcolor{blue}{#1}}}
\newcommand\luis[1]{{\textcolor{red}{#1}}}

\begin{center}

{\Large\bf Higher-Page Jacobian and Albanese Tori}

\end{center}

\begin{center}

{\large Dan Popovici and Luis Ugarte}

\end{center}

\vspace{1ex}

\noindent{\small{\bf Abstract.} We construct, through Hodge-theoretical methods, what we call the $E_r$-Jacobian torus, the $E_r$-Albanese torus and the $E_r$-Albanese map of any compact complex manifold that is either {\it page-(r-1)-$\partial\bar\partial$} or {\it $E_r$-sGG}. These two classes of manifolds, the former of which is contained in the latter, have been introduced recently by both authors jointly with J. Stelzig, respectively by the first-named author. A Hodge theory is also developed for the latter class of manifolds. We then apply our results to give structure theorems and an identity of algebraic dimensions in terms of the $E_r$-Albanese map and torus. Other applications result in cohomological and metrical theorems for the $6$-dimensional sphere when it is equipped either with a hypothetical complex structure or with the complex structures very recently claimed to exist in the literature. For example, we show that in the latter case no {\it strongly Gauduchon} metric exists on $S^6$.}

\vspace{1ex}

\section{Introduction}\label{section:introd} In this paper, we use the Fr\"olicher spectral sequence of a compact complex manifold to construct $E_r$-analogues of the Jacobian and Albanese tori, as well as the Albanese map, in terms of certain cohomologies, when the manifold belongs to one of two recently introduced classes. This leads to structure theorems for the manifolds $X$ in these classes as fibrations defined by what we call the {\it $E_r$-Albanese map} over the {\it $E_r$-Albanese torus} $Alb_r(X)$ that we construct through Hodge-theoretical methods. Our study relies on the interplay between two kinds of properties: {\it cohomological} and {\it metrical}.
 
\vspace{1ex}

Let $X$ be a compact complex manifold with $\mbox{dim}_\C X =n\geq 2$. In [PSU20a], [PSU20b] and [PSU20c], for any integer $r\geq 1$, we introduced and investigated the class of {\bf page-(r-1)-$\partial\bar\partial$-manifolds}. We say that $X$ has this property if the analogue of the Hodge decomposition, in which the Dolbeault cohomology spaces have been replaced by those on the $r$-th page of the Fr\"olicher spectral sequence of $X$, holds. (See [PSU20a, Theorem and Definition 1.2], recalled here as Definition \ref{Def:page-1-dd-bar}, for a precise definition.) When $r=1$, {\it page-$0$-$\partial\bar\partial$-manifolds} coincide with the classical {\it $\partial\bar\partial$-manifolds}.

On the other hand, a generalisation of strongly Gauduchon (sG) metrics was introduced in [Pop19, Definition 3.1]. A given Hermitian metric $\omega$ (identified with the corresponding $C^\infty$ positive definite $(1,\,1)$-form) on $X$ is said to be an {\bf $E_r$-sG metric} if $\partial\omega^{n-1}$ is $E_r$-exact. This means that $\partial\omega^{n-1}$ induces the zero cohomology class on the $r$-th page of the Fr\"olicher spectral sequence of $X$. (The notions of $E_1$-sG and sG metrics coincide.) Moreover, $X$ is said to be an {\bf $E_r$-sGG manifold} if every Gauduchon metric $\omega$ on $X$ is $E_r$-sG. These manifolds generalise the {\bf sGG manifolds} introduced and studied in [PU]: the $E_1$-sGG condition coincides with the sGG condition. (See Definition \ref{Def:sG_definitions} for a rundown on these and other metrical notions.) 

A key link between the cohomological and metrical points of view is provided by the following implications that hold for every integer $r\geq 1$: \begin{eqnarray*} X \hspace{1ex} \mbox{is a {\bf page-(r-1)-$\partial\bar\partial$-manifold}} & \implies & X \hspace{1ex} \mbox{is a {\bf page-r-$\partial\bar\partial$-manifold}} \\
  \Downarrow & & \Downarrow \\
  X \hspace{1ex} \mbox{is an {\bf $E_r$-sGG manifold}} & \implies & X \hspace{1ex} \mbox{is an {\bf $E_{r+1}$-sGG manifold}}.\end{eqnarray*}

In $\S$\ref{section:page_1}, we use the Hodge theory of {\it page-(r-1)-$\partial\bar\partial$-manifolds} $X$, developed in [PSU20a], to define the {\bf $E_r$-Jacobian torus} (Definition \ref{Def:E_r-Jacobi-torus}) and the {\bf $E_r$-Albanese torus} (Definition \ref{Def:E_r-Albanese-torus}) of $X$ as \begin{eqnarray*}Jac_r(X):= E_r^{0,\,1}(X)\Big/\mbox{Im}^{(r)}H^1(X,\,\Z) \hspace{5ex} \mbox{and} \hspace{5ex} Alb_r(X):= E_r^{n-1,\,n}(X)\Big/\mbox{Im}^{(r)}H^{2n-1}(X,\,\Z),\end{eqnarray*} where $\mbox{Im}^{(r)}$ stands for the image of the respective integral cohomology group under a certain canonical map defined Hodge-theoretically.

Both are (compact) complex tori of dimension half the first Betti number of $X$: \begin{eqnarray*}g_r = g_r(X):=\mbox{dim}_\C E_r^{0,\,1}(X) = \frac{1}{2}\,b_1(X).\end{eqnarray*} We then go on to construct (cf. Definition \ref{Def:E_r-Albanese-map}) the  {\bf $E_r$-Albanese map} $A_r:X\longrightarrow Alb_r(X)$ of any {\it page-(r-1)-$\partial\bar\partial$-manifold} $X$. This map is then proved to be a {\it holomorphic} (cf. Proposition \ref{Prop:E_r-Albanese-map_hol}). When $X$ is, moreover, {\bf complex parallelisable}, its  $E_r$-Albanese map $A_r:X\longrightarrow Alb_r(X)$ is a {\it surjective holomorphic submersion} (cf. Corollary \ref{Cor:A_r_surjective-submersion_c-par-page-ddbar}). This fact implies, in turn, a comparison between the algebraic dimensions of $X$ and $Alb_r(X)$ (cf. Corollary \ref{Cor:A_r_c-par-page-ddbar_alg-dim_ineq}).

\vspace{1ex}

These results are generalised to the class of {\bf $E_r$-sGG manifolds} in $\S$\ref{section:E_r-sGG}. Unlike those of $\S$\ref{section:page_1}, a Hodge theory was not known for these manifolds, so we developed one in $\S$\ref{subsection:Hodge_E_r-sGG}. The results (Theorem \ref{The:E_r-sGG_Hodge-decomp} and Corollaries \ref{Cor:E_r-sGG_Hodge-decomp} and \ref{Cor:E_r-sGG_b_1_2e_r-01}), which generalise those obtained in [PU] for the case $r=1$, can be summed up in the following

\begin{The}\label{The:introd_E_r-sGG_Hodge-decomp} Let $X$ be a compact complex manifold with $\mbox{dim}_\C X=n$. Fix $r\in\N^\star$.

The following conditions are equivalent:

\vspace{1ex}  

$(a)$\, $X$ is an {\bf $E_r$-sGG manifold};

\vspace{1ex}  

$(b)$\, $X$ has a {\bf weak Hodge decomposition in degree $1$ at the $E_r$ level} in the sense that the canonical $\C$-linear map: \begin{eqnarray*}F_r:H^1_{dR}(X,\,\C)\longrightarrow E_r^{0,\,1}(X)\oplus\overline{E_r^{0,\,1}(X)}, \hspace{5ex} \{\alpha\}_{dR}\longmapsto\bigg(\{\alpha^{0,\,1}\}_{E_r},\,\overline{\{\overline{\alpha^{1,\,0}}}\}_{E_r}\bigg),\end{eqnarray*} is {\bf well defined} and {\bf bijective};

\vspace{1ex}  

$(c)$\, $b_1(X) = 2\,e_r^{0,\,1}(X)$, where $b_1(X)$ is the first Betti number of $X$ and $e_r^{0,\,1}(X)=\mbox{dim}_\C E_r^{0,\,1}(X)$.

\end{The}  

\vspace{1ex}

These descriptions enable us to define the {\bf $E_r$-Jacobian torus} and the {\bf $E_r$-Albanese torus} (Definition \ref{Def:E_r-Jacobi-Albanese-tori_E_r-sGG}) of any {\it $E_r$-sGG manifold} $X$ as \begin{eqnarray*}Jac_r(X):= E_r^{0,\,1}(X)\Big/\mbox{Im}^{(r)}H^1(X,\,\Z) \hspace{3ex}\mbox{and}\hspace{3ex} Alb_r(X):= \overline{E_r^{n,\,n-1}(X)}\Big/\mbox{Im}^{(r)}H^{2n-1}(X,\,\Z),\end{eqnarray*} where $\mbox{Im}^{(r)}$ stands for the image of the respective integral cohomology group under a certain canonical map defined Hodge-theoretically. As for {\it page-(r-1)-$\partial\bar\partial$-manifolds}, both are (compact) complex tori of dimension $g_r(X): =\mbox{dim}_\C E_r^{0,\,1}(X) = (1/2)\,b_1(X)$. The {\bf $E_r$-Albanese map} $A_r:X\longrightarrow Alb_r(X)$ of any {\it $E_r$-sGG manifold} is then constructed in $\S$\ref{subsection:E_r-Albanese-map_E_r-sGG} and subsequently shown to have properties similar to those of its {\it page-(r-1)-$\partial\bar\partial$} counterpart. Some of these properties (cf. Proposition \ref{Prop:E_r-Albanese-map_hol_E_r_sGG} and Corollaries \ref{Cor:A_r_surjective-submersion_c-par-E_r-sGG} and \ref{Cor:A_r_E_r-sGG_alg-dim_eq}) are summed up in

\begin{The}\label{The:introd_E_r-Albanese_prop} Let $r\in\N^\star$ and let $X$ be a compact complex {\bf $E_r$-sGG manifold} with $\mbox{dim}_\C X =n$.

\vspace{1ex}

(a)\, The $E_r$-Albanese map $A_r:X\longrightarrow Alb_r(X)$ of $X$ is {\bf holomorphic}.

\vspace{1ex}

(b)\, If, moreover, $X$ is {\bf complex parallelisable}, then its $E_r$-Albanese map $A_r:X\longrightarrow Alb_r(X)$ is a holomorphic {\bf surjective submersion} and the following equality of algebraic dimensions holds: \begin{eqnarray*}\label{eqn:A_r_E_r-sGG_alg-dim_eq}a\bigg(Alb_r(X)\bigg)= a(X).\end{eqnarray*}

\end{The}

\vspace{2ex}

We then apply our results in a number of ways, two of which we now mention.

\vspace{1ex}

{\bf (I)}\, It was proved in Theorem 4.7 of [PSU20a] that all {\it complex parallelisable nilmanifolds} (of any dimension) are page-$1$-$\partial\bar\partial$-manifolds. (For the reader's convenience, we recall that the Iwasawa manifold $I^{(3)}$ is a complex parallelisable nilmanifold.) This result was generalised in [KS23, Theorem A]: all {\it complex parallelisable solvmanifolds} are page-$1$-$\partial\bar\partial$-manifolds. Alongside this, our Corollary \ref{Cor:A_r_surjective-submersion_c-par-page-ddbar} (or the above Theorem \ref{The:introd_E_r-Albanese_prop}) gives:

\begin{Cor}\label{Cor:A_r_surjective-submersion_c-par_solvmanifolds} If $X$ is a compact {\bf complex parallelisable solvmanifold}, its $E_2$-Albanese map $A_2:X\longrightarrow Alb_2(X)$ is a holomorphic {\bf surjective submersion}.

\end{Cor}

\vspace{1ex}

{\bf (II)}\, In $\S$\ref{subsubsection:S6}, we undertake a study of both cohomological and metrical properties of the $6$-dimensional sphere $S^6$. We first analyse $S^6$ when it is equipped with any hypothetical complex structure and then when it is equipped with the complex structure claimed in [A26] and further investigated in [Eng26]. In addition to results obtained in earlier references such as [Uga00] and [Gra97], our study produces new results, some of which are of a different nature, in line with the other notions that are dealt with in the present paper.

We summarise some of them (Proposition \ref{Prop:S6-precise-r-2-or-3}, Proposition \ref{Prop:S6-SKT-metrics}, Remark \ref{Rem:S6-h11-ne-0}, Theorem \ref{The:non-sG_S6}) in the following

\begin{The}\label{the:introd_S6} Let $X$ be $S^6$ endowed with any hypothetical complex structure $J$. Then:

\vspace{1ex}

(i)\, $X$ is a {\bf page-$2$-$\partial\bar\partial$-manifold}, hence also an {\bf $E_3$-sGG manifold};

\vspace{1ex}

(ii)\, if $X$ supports an {\bf SKT metric}, then $X$ is a {\bf page-$1$-$\partial\bar\partial$-manifold};

\vspace{1ex}

(iii)\, if $Y$ is the compact complex $3$-dimensional manifold $(S^6,\,J_A)$, where $J_A$ is the complex structure claimed in [A26] to exist on $S^6$, $Y$ is a {\bf page-$1$-$\partial\bar\partial$-manifold} and it carries {\bf no strongly Gauduchon} (hence {\bf no balanced}) metric.

\end{The}

\vspace{2ex}

New families of complex structures on $S^6$, that are not biholomorphic to the Alp\"oge-Claude examples of [A26, Eng26], are constructed in [Via26] and [LR26]. Very little information on the Hodge numbers of these complex structures is provided in those references. However, we point out that the properties and results obtained in this paper for $X=(S^6,\,J)$ are valid for them as they hold for any complex structure on $S^6$. Moreover, since for all the new complex structures claimed to exist on $S^6$ there is a surjective holomorphic map onto $\Proj^1$, we can still apply our Theorem~\ref{The:non-sG_S6} to conclude that none of them carries any strongly Gauduchon (or balanced) metric.

\vspace{2ex}

We wind this discussion up with the following

\begin{Question}\label{introd_Question:S6_SKT} Does the compact complex manifold $Y=(S^6,\,J_A)$ of [A26] admit an {\bf SKT metric}?

\end{Question}

We will see in Proposition \ref{Prop:S6-SKT-metrics} that the {\it page-$1$-$\partial\bar\partial$} property is necessary for the existence of an SKT metric on $S^6$ equipped with any complex structure. We will then see in Remark \ref{Rem:S6-h11-ne-0} that this necessary condition is satisfied by the manifold $Y=(S^6,\,J_A)$ of [A26]. These results thrust Question \ref{introd_Question:S6_SKT} to the fore.

\section{Case of page-(r-1)-$\partial\bar\partial$-manifolds}\label{section:page_1} Let $X$ be a compact complex manifold with $\mbox{dim}_\C X =n$. The following notion was introduced in [PSU20a, Theorem and Definition 1.2].

\begin{Def}\label{Def:page-1-dd-bar} One says that $X$ is a {\bf page-(r-1)-$\partial\bar\partial$-manifold} if the following two conditions are satisfied:

\vspace{1ex}

(i)\, for all $p,q\in\{0,\dots , n\}$, every class $\mathfrak{c}\in E_r^{p,\,q}(X)$ in the $\C$-vector space of bidegree $(p,\,q)$ on the $r$-th page of the Fr\"olicher spectral sequence of $X$ can be represented by a {\bf $d$-closed} form $\alpha^{p,\,q}\in\mathfrak{c}$ of pure type $(p,\,q)$;

\vspace{1ex}

(ii)\, for every $k\in\{0,\dots , 2n\}$, the linear map \begin{eqnarray}\label{eqn:E_r-Hodge-decomp}T_r^{(k)}:\bigoplus\limits_{p+q=k}E_r^{p,\,q}(X)\longrightarrow H^k_{dR}(X,\,\C),  \hspace{5ex} \sum\limits_{p+q=k}\{\alpha^{p,\,q}\}_{E_r}\longmapsto\bigg\{\sum\limits_{p+q=k}\alpha^{p,\,q} \bigg\}_{dR},\end{eqnarray} is {\bf well defined} and an {\bf isomorphism}, where, for each bidegree $(p,\,q)$, a representative $\alpha^{p,\,q}$ of type $(p,\,q)$ of the corresponding $E_r$-class has been chosen such that $d\alpha^{p,\,q} = 0$.

In this case, we denote by $S_r^{(k)}$ the inverse of $T_r^{(k)}$.
  
\end{Def}  

Well-definedness means, in particular, that the map $T_r^{(k)}$ is independent of the choices of $d$-closed $(p,\,q)$-form representatives $\alpha^{p,\,q}$ of their respective $E_r$-classes. The canonical isomorphisms (\ref{eqn:E_r-Hodge-decomp}), induced by the identity map on differential forms, are called the {\bf $E_r$-Hodge decomposition} property of $X$. As proved in Corollary 2.11 of [PSU20a], this implies the {\bf $E_r$-Hodge symmetry} property (cf. ([PSU20a, Definition 2.8]), namely the fact that the linear map \begin{eqnarray}\label{eqn:E_r-Hodge-sym}E_r^{p,\,q}(X)\ni\{\alpha^{p,\,q}\}_{E_r}\mapsto\overline{\{\overline{\alpha^{p,\,q}}\}_{E_r}}\in \overline{E_r^{q,\,p}(X)}\end{eqnarray} is {\bf well-defined} (in the sense that it does not depend on the choices of $d$-closed representatives $\alpha^{p,\,q}$ of the classes $\{\alpha^{p,\,q}\}_{E_r}$) and {\bf bijective}.

\vspace{1ex}

Let us recall the following characterisation of {\bf page-(r-1)-$\partial\bar\partial$-manifolds}:

\begin{The}([PSU20a, Theorem 2.9])\label{The:page_r-1_pure-characterisation} Let $X$ be a compact complex manifold with $\mbox{dim}_\C X=n$. Fix an arbitrary $r\in\N$. Then, the following two conditions are equivalent:

\vspace{1ex}

$(1)$\, $X$ is a {\bf page-(r-1)-$\partial\bar\partial$-manifold};

\vspace{1ex}

$(2)$\, the Fr\"olicher spectral sequence of $X$ {\bf degenerates at $E_r$} (a property denoted by $E_r(X) = E_\infty(X)$) and the de Rham cohomology of $X$ is {\bf pure}

\end{The}

\vspace{1ex}

By the {\it de Rham cohomology} of $X$ being {\it pure} we mean that \begin{eqnarray}\label{eqn:de-Rham_pure_def}H^k_{dR}(X,\,\C) = \bigoplus_{p+q=k}H^{p,\,q}_{dR}(X), \hspace{5ex} k\in\{0,\dots , 2n\},\end{eqnarray} where, for every bidegree $(p,\,q)$, $H^{p,\,q}_{dR}(X)$ denotes the $\C$-vector subspace of the de Rham cohomology space $H^{p+q}_{dR}(X,\,\C)$ consisting of de Rham cohomology classes of degree $p+q$ that are representable by pure-type ($d$-closed) $(p,\,q)$-forms.  

\vspace{2ex}

On the other hand, on every compact complex manifold (no page-(r-1)-$\partial\bar\partial$ or any other assumption is necessary), there is a {\bf canonical Serre-type duality} on every page of the Fr\"olicher spectral sequence:

\begin{The}([PSU20b, Theorem 1.4])\label{The:duality_E_r} Let $X$ be a compact complex manifold with $\mbox{dim}_\C X=n$. Fix an arbitrary positive integer $r\in\N$. For every $p,q\in\{0,\dots , n\}$, the canonical bilinear pairing: \begin{eqnarray*} E_r^{p,\,q}(X)\times E_r^{n-p,\,n-q}(X)\longrightarrow\C, \hspace{3ex} (\{\alpha\}_{E_r},\,\{\beta\}_{E_r})\mapsto \int\limits_X\alpha\wedge\beta,\end{eqnarray*} is {\bf well defined} and {\bf non-degenerate}.

\end{The}

\subsection{The $E_r$-Jacobian torus of a page-(r-1)-$\partial\bar\partial$-manifold}\label{subsection:E_r-Jacobi-torus_page_r-1}

Now, recall that for every compact complex manifold $X$, the image of the canonical map $H^1(X,\,\Z)\longrightarrow H^1(X,\,\R)$ is a lattice. Hence, the image of the canonical map $J_1:H^1(X,\,\Z)\longrightarrow H^1_{dR}(X,\,\R)$ is a lattice. On the other hand, if $X$ is {\it page-(r-1)-$\partial\bar\partial$}, it follows from the isomorphisms (\ref{eqn:E_r-Hodge-decomp}) and (\ref{eqn:E_r-Hodge-sym}) in degree $k=1$ that the composition \begin{eqnarray*}p_r^{0,\,1}\circ S_r^{(1)}\circ j_1 : H^1_{dR}(X,\,\R)\longrightarrow E_r^{0,\,1}(X)\end{eqnarray*} is a canonical {\it $\R$-isomorphism}, where in \begin{eqnarray*}H^1_{dR}(X,\,\R)\xhookrightarrow{j_1}H^1_{dR}(X,\,\C)\xrightarrow[\simeq]{S_r^{(1)}}E_r^{1,\,0}(X)\oplus E_r^{0,\,1}(X)\xtwoheadrightarrow{p_r^{0,\,1}}E_r^{0,\,1}(X)\end{eqnarray*} $j_1$ is the canonical injection, $S_r^{(1)}$ is the isomorphism inverse to the $E_r$-Hodge decomposition $T_r^{(1)}$, the vector spaces $E_r^{1,\,0}(X)$ and $E_r^{0,\,1}(X)$ are canonically conjugate to each other and $p_r^{0,\,1}$ is the projection onto the second factor.

Putting these two pieces of information together, we get

\begin{Cor}\label{Cor:H1-Z_E_r_0-1_lattice} If $X$ is a compact complex {\bf page-(r-1)-$\partial\bar\partial$-manifold}, the image of the map \begin{eqnarray*}p_r^{0,\,1}\circ S_r^{(1)}\circ j_1\circ J_1 : H^1(X,\,\Z)\longrightarrow E_r^{0,\,1}(X),\end{eqnarray*} that we denote by $\mbox{Im}^{(r)}H^1(X,\,\Z)$, is a {\bf lattice} in $E_r^{0,\,1}(X)$.

\end{Cor}   

This enables us to introduce the following

\begin{Def}\label{Def:E_r-Jacobi-torus} Let $r\in\N^\star$ and let $X$ be a compact complex {\bf page-(r-1)-$\partial\bar\partial$-manifold}. Set $g_r:=\mbox{dim}_\C E_r^{0,\,1}(X)$. The $g_r$-dimensional complex torus \begin{eqnarray*}Jac_r(X):= E_r^{0,\,1}(X)\Big/\mbox{Im}^{(r)}H^1(X,\,\Z)\end{eqnarray*} is called the {\bf $E_r$-Jacobian torus} of $X$.

\end{Def}

\vspace{2ex}

Note that, when $r=1$ and $X$ is a {\bf $\partial\bar\partial$-manifold}, the $E_r$-Jacobian torus of $X$ defined above coincides with the usual {\bf Jacobian torus} defined Hodge-theoretically (see [Dem97, VI, $\S 9.1$]) by \begin{eqnarray}\label{eqn:Jacobi-torus_standard-dd-bar}Jac(X):= H^{0,\,1}(X)\Big/\mbox{Im}\, H^1(X,\,\Z).\end{eqnarray}

Note also that our $E_r$-Jacobian torus might be ``smaller'' than the classical one. Indeed, thanks to all $d_r$ maps on the $r$-th page of the Fr\"olicher spectral sequence being of bidegree $(r,\,-r+1)$ for every $r\geq 1$, we have the following complex on the $(r-1)$-st page when $r\geq 2$: \begin{eqnarray*}\{0\} = E_{r-1}^{-r+1,\,r-1}(X)\xrightarrow{d_{r-1}^{-r+1,\,r-1}} E_{r-1}^{0,\,1}(X)\xrightarrow{d_{r-1}^{0,\,1}} E_{r-1}^{r-1,\,-r+3}(X)\longrightarrow\cdots\end{eqnarray*} This yields the inclusion: \begin{eqnarray*}E_r^{0,\,1}(X) = \ker d_{r-1}^{0,\,1}\subset E_{r-1}^{0,\,1}(X),  \hspace{5ex} r\geq 2.\end{eqnarray*}

\subsection{The $E_r$-Albanese torus of a page-(r-1)-$\partial\bar\partial$-manifold}\label{subsection:E_r-Albanese-torus_page_r-1}

Similarly, in degree $2n-1$, for any compact complex $n$-dimensional manifold $X$, the image of the canonical map $J_{2n-1}:H^{2n-1}(X,\,\Z)\longrightarrow H^{2n-1}_{dR}(X,\,\R)$ is a lattice. When $X$ is {\it page-(r-1)-$\partial\bar\partial$}, the isomorphisms (\ref{eqn:E_r-Hodge-decomp}) and (\ref{eqn:E_r-Hodge-sym}) in degree $k=2n-1$ imply that the composition \begin{eqnarray*}p_r^{n-1,\,n}\circ S_r^{(2n-1)}\circ j_{2n-1} : H^{2n-1}_{dR}(X,\,\R)\longrightarrow E_r^{n-1,\,n}(X)\end{eqnarray*} is a canonical {\it $\R$-isomorphism}, where in \begin{eqnarray*}H^{2n-1}_{dR}(X,\,\R)\xhookrightarrow{j_{2n-1}}H^{2n-1}_{dR}(X,\,\C)\xrightarrow[\simeq]{S_r^{(2n-1)}}E_r^{n,\,n-1}(X)\oplus E_r^{n-1,\,n}(X)\xtwoheadrightarrow{p_r^{n-1,\,n}}E_r^{n-1,\,n}(X)\end{eqnarray*} $j_{2n-1}$ is the canonical injection, $S_r^{(2n-1)}$ is the isomorphism inverse to the $E_r$-Hodge decomposition $T_r^{(2n-1)}$, the vector spaces $E_r^{n,\,n-1}(X)$ and $E_r^{n-1,\,n}(X)$ are canonically conjugate to each other and $p_r^{n-1,\,n}$ is the projection onto the second factor.

From this we get, as in degree $1$, the following

\begin{Cor}\label{Cor:H_2n-1-Z_E_r_n-1-n_lattice} If $X$ is a compact complex $n$-dimensional {\bf page-(r-1)-$\partial\bar\partial$-manifold}, the image of the map \begin{eqnarray*}p_r^{n-1,\,n}\circ S_r^{(2n-1)}\circ j_{2n-1}\circ J_{2n-1} : H^{2n-1}(X,\,\Z)\longrightarrow E_r^{n-1,\,n}(X),\end{eqnarray*} that we denote by $\mbox{Im}^{(r)}H^{2n-1}(X,\,\Z)$, is a {\bf lattice} in $E_r^{n-1,\,n}(X)$.

\end{Cor}   

This enables us to introduce the following

\begin{Def}\label{Def:E_r-Albanese-torus} Let $r\in\N^\star$ and let $X$ be a compact complex {\bf page-(r-1)-$\partial\bar\partial$-manifold} with $\mbox{dim}_\C X =n$. Set $g_r:=\mbox{dim}_\C E_r^{0,\,1}(X)$. The $g_r$-dimensional complex torus \begin{eqnarray*}Alb_r(X):= E_r^{n-1,\,n}(X)\Big/\mbox{Im}^{(r)}H^{2n-1}(X,\,\Z)\end{eqnarray*} is called the {\bf $E_r$-Albanese torus} of $X$.

\end{Def}

\vspace{2ex}

Note that, when $r=1$ and $X$ is a {\bf $\partial\bar\partial$-manifold}, the $E_r$-Albanese torus of $X$ defined above coincides with the usual {\bf Albanese torus} defined Hodge-theoretically (see [Dem97, VI, $\S 9.2$]) by \begin{eqnarray}\label{eqn:Albanese-torus_standard-dd-bar}Alb(X):= H^{n-1,\,n}(X)\Big/\mbox{Im}\, H^{2n-1}(X,\,\Z).\end{eqnarray} Recall that the Albanese torus of an arbitrary compact complex manifold $X$ was defined by Blanchard [Bla58] (see also [Uen75, chapter IV, $\S.9$]), but it can only be given the Hodge-theoretical description (\ref{eqn:Albanese-torus_standard-dd-bar}) when $X$ is a $\partial\bar\partial$-manifold.

One of the goals of this paper is to extend the Hodge-theoretical descriptions of the Jacobian and Albanese tori, possibly by shrinking Blanchard's classical one, to broader classes of compact complex manifolds, defined by weaker conditions than the $\partial\bar\partial$-condition and by considering higher pages of the Fr\"olicher spectral sequence of $X$.

\begin{Rem}\label{Rem:g_r_tori-duality} (i)\, For any compact complex $n$-dimensional manifold $X$, the duality statement in Theorem \ref{The:duality_E_r} for $(p,\,q)=(1,\,0)$ and $(p,\,q)=(0,\,1)$ yields {\bf canonical isomorphisms}: \begin{eqnarray}\label{eqn:duality_1-0_0-1}E_r^{n-1,\,n}(X)\simeq E_r^{1,\,0}(X)^\star \hspace{5ex} \mbox{and} \hspace{5ex}  E_r^{0,\,1}(X) \simeq E_r^{n,\,n-1}(X)^\star.\end{eqnarray}

  (ii)\, If, morever, $X$ is a {\bf page-(r-1)-$\partial\bar\partial$-manifold}, the $E_r$-Hodge symmetry property (\ref{eqn:E_r-Hodge-sym}) for $(p,\,q)=(n-1,\,n)$, together with (i), yields a {\bf canonical isomorphism}: \begin{eqnarray}\label{eqn:duality-symmetry_conjugation_0-1}E_r^{n-1,\,n}(X)\simeq\bigg(\overline{E_r^{0,\,1}(X)}\bigg)^\star\end{eqnarray} and the following {\bf canonical sesquilinear duality} of $\C$-vector spaces: \begin{eqnarray*} E_r^{0,\,1}(X)\times E_r^{n-1,\,n}(X)\longrightarrow\C, \hspace{3ex} (\{\alpha\}_{E_r},\,\{\beta\}_{E_r})\mapsto \int\limits_X\alpha\wedge\overline\beta.\end{eqnarray*}

    Coupled with the Poincar\'e duality between $H^1(X,\,\Z)$ and $H^{2n-1}(X,\,\Z)$, this enables us to say that the $E_r$-Jacobian torus $Jac_r(X)$ and the $E_r$-Albanese torus $Alb_r(X)$ are naturally sesquilinearly dual to each other.

    Furthermore, combining this with the $E_r$-Hodge decomposition property (\ref{eqn:E_r-Hodge-decomp}) of $X$, we get:  \begin{eqnarray*}g_r = g_r(X):= \mbox{dim}_\C E_r^{0,\,1}(X) = \mbox{dim}_\C E_r^{n-1,\,n}(X) = \frac{1}{2}\,b_1(X),\end{eqnarray*} where $b_1 = b_1(X)$ is the first Betti number of $X$.

\vspace{1ex}    

(iii)\, Recall that, for any compact complex $n$-dimensional manifold $X$, the Poincar\'e duality gives the {\bf canonical isomorphism} \begin{eqnarray}\label{eqn:Poincare-duality_deg-1_Z}H_1(X,\,\Z)\longrightarrow H^{2n-1}(X,\,\Z), \hspace{5ex} [\xi]\longmapsto\{I_\xi\}_{dR},\end{eqnarray} where for every loop $\xi$ in $X$, $I_\xi:=[\xi]$ stands for the current of integration on $\xi$.

Thus, thanks to (\ref{eqn:duality-symmetry_conjugation_0-1}) and (\ref{eqn:Poincare-duality_deg-1_Z}), when $X$ is a {\bf page-(r-1)-$\partial\bar\partial$-manifold} we get the following alternative description of the $E_r$-Albanese torus of $X$: \begin{eqnarray}\label{eqn:E_r-Albanese_alternative}Alb_r(X)=\bigg(\overline{E_r^{0,\,1}(X)}\bigg)^\star\Big/\mbox{Im}^{(r)}H_1(X,\,\Z),\end{eqnarray} where $\mbox{Im}^{(r)}H_1(X,\,\Z)$ is the image of the first homology group $H_1(X,\,\Z)$ under the canonical map: \begin{eqnarray}\label{eqn:map_H_1_E_r_0-1_bar_star}H_1(X,\,\Z)\longrightarrow\bigg(\overline{E_r^{0,\,1}(X)}\bigg)^\star, \hspace{5ex} [\xi]\longmapsto\widetilde{I}_\xi:=\bigg(\overline{E_r^{0,\,1}(X)}\ni\overline{\{\alpha^{0,\,1}\}_{E_r}}\longmapsto\int_\xi\{\overline{\alpha^{0,\,1}}\}_{dR}\in\C\bigg),\end{eqnarray} defined using $(0,\,1)$-type representatives $\alpha^{0,\,1}$ of classes $\mathfrak{c}\in E_r^{0,\,1}(X)$ such that $d\alpha^{0,\,1} = 0$. Such representatives exist thanks to the page-(r-1)-$\partial\bar\partial$-hypothesis on $X$ (see (i) of Definition \ref{Def:page-1-dd-bar}) and the map (\ref{eqn:map_H_1_E_r_0-1_bar_star}) is independent of the choice of $d$-closed representative of type $(0,\,1)$ of each class in $ E_r^{0,\,1}(X)$.

Indeed, the map $T_r^{(1)}$ is independent of these choices. Furthermore, for each $d$-closed representative $\alpha^{0,\,1}$ of a class in $ E_r^{0,\,1}(X)$, we get $d\overline{\alpha^{0,\,1}} = 0$, by conjugation, so $\overline{\alpha^{0,\,1}}$ represents a de Rham class $\{\overline{\alpha^{0,\,1}}\}_{dR}\in H^1_{dR}(X,\,\C)$. Finally, the last integral in (\ref{eqn:map_H_1_E_r_0-1_bar_star}) can be written as \begin{eqnarray*}\int_\xi\{\overline{\alpha^{0,\,1}}\}_{dR} = \int\limits_X I_\xi\wedge\overline{\alpha^{0,\,1}}.\end{eqnarray*} Hence, it is well defined (in the sense that it is independent of the choice of representative $\overline{\alpha^{0,\,1}}$ of the de Rham class $\{\overline{\alpha^{0,\,1}}\}_{dR}$ -- because, thanks to $\xi$ being a loop,  the current of integration $I_\xi$ on it is $d$-closed) and depends only on the homotopy class $[\xi]\in H_1(X,\,\Z)$ of the loop $\xi$.

\end{Rem}

\subsection{The $E_r$-Albanese map of a page-(r-1)-$\partial\bar\partial$-manifold}\label{subsection:E_r-Albanese-map_page_r-1} Using the description (\ref{eqn:E_r-Albanese_alternative}) of the $E_r$-Albanese torus, we make the following

\begin{Def}\label{Def:E_r-Albanese-map} Let $r\in\N^\star$ and let $X$ be a compact complex {\bf page-(r-1)-$\partial\bar\partial$-manifold} with $\mbox{dim}_\C X =n$. Fix a base point $a\in X$. For every point $x\in X$, let $\xi_x$ be a path in $X$ from $a$ to $x$ and let $\widetilde{I}_{\xi_x}\in\bigg(\overline{E_r^{0,\,1}(X)}\bigg)^\star$ be the linear functional defined as in (\ref{eqn:map_H_1_E_r_0-1_bar_star}) by \begin{eqnarray}\label{eqn:map_H_1_E_r_0-1_bar_star_paths}\overline{E_r^{0,\,1}(X)}\ni\overline{\{\alpha^{0,\,1}\}_{E_r}}\xmapsto{\widetilde{I}_{\xi_x}}\int_{\xi_x}\overline{\alpha^{0,\,1}}\in\C,\end{eqnarray} where, for each class $\mathfrak{c}\in E_r^{0,\,1}(X)$, $\alpha^{0,\,1}$ denotes a $d$-closed $(0,\,1)$-type representative of $\mathfrak{c}$. 

 The canonical map: \begin{eqnarray}\label{eqn:E_r-Albanese-map}A_r:X\longrightarrow Alb_r(X)=\bigg(\overline{E_r^{0,\,1}(X)}\bigg)^\star\Big/\mbox{Im}^{(r)}H_1(X,\,\Z), \hspace{5ex} x\longmapsto \widetilde{I}_{\xi_x} \hspace{3ex} \mbox{mod}\hspace{1ex} \mbox{Im}^{(r)}H_1(X,\,\Z),\end{eqnarray} is called the {\bf $E_r$-Albanese map} of $X$.

\end{Def}

\vspace{2ex}

A word of explanation is in order. For a path  $\xi_x$ from $a$ to $x$, the current of integration $I_{\xi_x}$ on $\xi_x$ is not $d$-closed when $x\neq a$. Therefore, the integral \begin{eqnarray*}\int_{\xi_x}\overline{\alpha^{0,\,1}} = \int\limits_X I_{\xi_x}\wedge\overline{\alpha^{0,\,1}}\end{eqnarray*} depends on the choice of path $\xi_x$ from $a$ to $x$ despite the fact that $d\overline{\alpha^{0,\,1}} = 0$. It also depends on the choice of representative $\overline{\alpha^{0,\,1}}$ of the de Rham class $\{\overline{\alpha^{0,\,1}}\}_{dR}$, hence on the choice of $d$-closed $(0,\,1)$-type representative $\alpha^{0,\,1}$ of its class $\{\alpha^{0,\,1}\}_{E_r}\in E_r^{0,\,1}(X)$. Nevertheless, the $E_r$-Albanese map $A_r:X\longrightarrow Alb_r(X)$ is well defined because for any other path $\eta_x$ from $a$ to $x$, the path $\eta_x^{-1}\,\xi_x$ is a loop, hence $\widetilde{I}_{\eta_x^{-1}\,\xi_x}\in\mbox{Im}^{(r)}H_1(X,\,\Z)$. Explicitly, we have: \begin{eqnarray*}\int_{\xi_x}\overline{\alpha^{0,\,1}}- \int_{\eta_x}\overline{\alpha^{0,\,1}} = \int_{\eta_x^{-1}\,\xi_x}\overline{\alpha^{0,\,1}} = \int\limits_X I_{\eta_x^{-1}\,\xi_x}\wedge\overline{\alpha^{0,\,1}}\end{eqnarray*} and the last integral depends only on the homology class $[\eta_x^{-1}\,\xi_x]\in H_1(X,\,\Z)$ of the loop $\eta_x^{-1}\,\xi_x$ and on the de Rham class $\{\overline{\alpha^{0,\,1}}\}_{dR}$.

\vspace{2ex}

We now give a different description of the $E_r$-Albanese map, following again in our generalised setting the presentation of the classical case given in [Dem97, chapter VI - $\S9$].

Recall that, by the Hodge symmetry property (\ref{eqn:E_r-Hodge-sym}) of compact complex page-(r-1)-$\partial\bar\partial$-manifolds $X$, $\overline{E_r^{0,\,1}(X)}$ is canonically isomorphic to $E_r^{1,\,0}(X)$, the isomorphism being induced by complex conjugation of $d$-closed representatives of $E_r$-classes.

On the other hand, for every compact complex manifold $X$, $E_r^{p,\,0}(X)$ can be seen as the quotient of the $\C$-vector space of $E_r$-closed $(p,\,0)$-forms by the $\C$-vector space of $E_r$-exact $(p,\,0)$-forms on $X$ (see terminology in [PSU20b, Proposition 2.3] -- it applies to each $E_r^{p,\,q}(X)$ in any bidegree, in particular to bidegree $(p,\,0)$). For bidegree reasons, for every $r\geq 2$, a $(p,\,0)$-form is $E_r$-closed if and only if it is both $\partial$-closed and $\bar\partial$-closed, or equivalently, $d$-closed. On the other hand, if $r=1$ and $X$ is a $\partial\bar\partial$-manifold in the usual sense, every $\bar\partial$-closed $(p,\,0)$-form is also $\partial$-closed. Thus, for every value of $r\geq 1$, any $(p,\,0)$-type representative of any class in $E_r^{p,\,0}(X)$ on a page-(r-1)-$\partial\bar\partial$-manifold $X$ is $d$-closed. (The page-(r-1)-$\partial\bar\partial$-assumption on $X$ is not necessary when $r\geq 2$.) Meanwhile, in an obvious way, for every $p$, every class in $H_{\bar\partial}^{p,\,0}(X,\,\C)$ is a singleton. So, the elements of $H_{\bar\partial}^{p,\,0}(X,\,\C)$ identify with the $\bar\partial$-closed ($=$ holomorphic) smooth $(p,\,0)$-forms on $X$.  

\begin{Lem}\label{Lem:E_r_1-0_singleton_ker} Let $X$ be a compact complex manifold with $n:=\mbox{dim}_\C X\geq 2$. Fix an $r\geq 1$. Then:

\vspace{1ex}

(i)\, every class in $E_r^{1,\,0}(X)$ is a singleton;

\vspace{1ex}

(ii)\, $\displaystyle E_r^{1,\,0}(X) = \ker\bigg(d: H_{\bar\partial}^{1,\,0}(X,\,\C)\longrightarrow H_{\bar\partial}^{2,\,0}(X,\,\C)\bigg)\subset H_{\bar\partial}^{1,\,0}(X,\,\C)$ if $r\geq 2$.
  
\end{Lem}

\noindent {\it Proof.} (i)\, As mentioned above, this is immediate when $r=1$. Suppose therefore that $r\geq 2$. Let $\alpha^{1,\,0}$ be an arbitrary $E_r$-closed $(1,\,0)$-form on $X$. In this bidegree, this means that $d\alpha^{1,\,0}=0$. We have to show that $\alpha^{1,\,0}$ is $E_r$-exact if and only if $\alpha^{1,\,0}=0$. The ``if'' implication is obvious.

Suppose that $\alpha^{1,\,0}$ is $E_r$-exact. In this bidegree, this implies the existence of a $C^\infty$ function $f:X\to\C$ such that $\alpha^{1,\,0} = \partial f$ on $X$ (and a little more if $r\geq 2$, but we will not need the extra information). Taking $\bar\partial$, we get: \begin{eqnarray*}0=\bar\partial\alpha^{1,\,0} = \bar\partial\partial f  \hspace{5ex} \mbox{on}\hspace{1ex} X.\end{eqnarray*} Since $X$ is compact, this implies that $f$ is constant, so  $\alpha^{1,\,0} = \partial f =0$ on $X$.

\vspace{1ex}

(ii)\, As seen above, every class in either $H_{\bar\partial}^{1,\,0}(X,\,\C)$ or $E_r^{1,\,0}(X)$ is a singleton. Moreover, for each $r\geq 2$, the elements of $E_r^{1,\,0}(X)$ identify with the $d$-closed smooth $(1,\,0)$-forms on $X$. 

Therefore, a $C^\infty$ $(1,\,0)$-form $\alpha^{1,\,0}$ such that $\bar\partial\alpha^{1,\,0}=0$ (equivalently, the unique representative of a class in $H_{\bar\partial}^{1,\,0}(X,\,\C)$) lies in the kernel of $d$ if and only if $\alpha^{1,\,0}$ is the (unique) representative of a class in $E_r^{1,\,0}(X)$.  \hfill $\Box$

\vspace{2ex}

We get the following statement that will also be useful in explicit computations of the $E_r$-Albanese map of certain examples of page-(r-1)-$\partial\bar\partial$-manifolds. 

\begin{Prop}\label{Prop:E_r-Albanese-map_hol} Let $r\in\N^\star$ and let $X$ be a compact complex {\bf page-(r-1)-$\partial\bar\partial$-manifold} with $\mbox{dim}_\C X =n$. Let $\{\alpha_1^{1,\,0},\dots , \alpha_{g_r}^{1,\,0}\}$ be a $\C$-basis of $E_r^{1,\,0}(X)$.

  Then, the {\bf $E_r$-Albanese map} of $X$ identifies canonically with the map:  \begin{eqnarray}\label{eqn:E_r-Albanese-map_coordinates}\nonumber A_r:X &\longrightarrow & Alb_r(X)=\bigg(E_r^{1,\,0}(X)\bigg)^\star\Big/\mbox{Im}^{(r)}H_1(X,\,\Z)\simeq\C^{g_r}\Big/\Lambda, \\
    x & \longmapsto & \bigg(\int_{\xi_x}\alpha_1^{1,\,0},\dots , \int_{\xi_x}\alpha_{g_r}^{1,\,0}\bigg)   \hspace{3ex} \mbox{mod}\hspace{1ex} \Lambda,\end{eqnarray} where the lattice $\Lambda\subset\C^{g_r}$ is the group of periods of the $1$-forms $\alpha_1^{1,\,0},\dots , \alpha_{g_r}^{1,\,0}$, namely \begin{eqnarray*}\Lambda:=\bigg\{\bigg(\int_\gamma\alpha_1^{1,\,0},\dots , \int_\gamma\alpha_{g_r}^{1,\,0}\bigg)\,\bigg|\,[\gamma]\in\pi_1(X,\,a)\bigg\},\end{eqnarray*} and $\xi_x$ is any path in $X$ from the base point $a$ to $x$.

  In particular, the {\bf $E_r$-Albanese map} $A_r:X\longrightarrow Alb_r(X)$ is {\bf holomorphic}. 

\end{Prop}  

\noindent {\it Proof.} By the above discussion, $d\alpha_j^{1,\,0} = 0$ for all $j=1,\dots , g_r$. In particular, the $\alpha_j^{1,\,0}$ are holomorphic $1$-forms. Thus locally on $X$, for example on a sufficiently small simply connected open subset $U$ containing $x$, $\alpha_j^{1,\,0} = df_j$ for a {\it holomorphic} function $f_j$. Fix a point $x_0\in U$ and choose the paths $\xi_x$ by taking a fixed path from $a$ to $x_0$ and then a path inside $U$ from $x_0$ to $x$. We get: \begin{eqnarray*}\int_{\xi_x}\alpha_j^{1,\,0} = C_j + f_j(x) - f_j(x_0),  \hspace{5ex} j=1,\dots , g_r.\end{eqnarray*} This shows that the local representative of $A_r$ in $\C^{g_r}/\Lambda$ is holomorphic.   \hfill $\Box$

\vspace{2ex}

\begin{Cor}\label{Cor:A_r_surjective-submersion_c-par-page-ddbar} If $X$ is a compact {\bf complex parallelisable} {\bf page-(r-1)-$\partial\bar\partial$-manifold}, its $E_r$-Albanese map $A_r:X\longrightarrow Alb_r(X)$ is a holomorphic {\bf surjective submersion}.

\end{Cor}

\noindent {\it Proof.} By the complex parallelisability hypothesis, the holomorphic tangent bundle $T^{1,\,0}X$, hence also the holomorphic cotangent bundle $\Lambda^{1,\,0}T^\star X$, is (holomorphically) trivial. Therefore, setting $n:=\mbox{dim}_\C X$, there exist $n$ global holomorphic $(1,\,0)$-forms $\omega^1,\dots , \omega^n\in H^{1,\,0}_{\bar\partial}(X,\,\C)$ that are linearly independent at every point $x\in X$. In particular, they form a $\C$-basis of $H^{1,\,0}_{\bar\partial}(X,\,\C)$. (To see this last -- standard -- claim, note that $\omega^1,\dots , \omega^n\in H^{1,\,0}_{\bar\partial}(X,\,\C)$ are necessarily linearly independent in $H^{1,\,0}_{\bar\partial}(X,\,\C)$, so $h^{1,\,0}_{\bar\partial}(X)\geq n$. On the other hand, $H^{1,\,0}_{\bar\partial}(X,\,\C)\simeq H^0(X,\,\Omega^1_X) \simeq H^0(X,\,\C^n)$, the last isomorphism being a consequence of the triviality of the rank-n cotangent sheaf $\Omega^1_X$. Since $X$ is compact, we have $H^0(X,\,\C^n)\simeq\C^n$, so $h^{1,\,0}_{\bar\partial}(X) = n$.) 

We have already seen that $E_r^{1,\,0}(X)$ identifies canonically with the $\C$-vector space of $d$-closed $C^\infty$ $(1,\,0)$-forms on $X$ (see Lemma \ref{Lem:E_r_1-0_singleton_ker}). Note that this holds even for $r=1$ thanks to our page-(r-1)-$\partial\bar\partial$-assumption (which becomes the standard $\partial\bar\partial$-assumption in that case).

Let $\{\alpha_1^{1,\,0},\dots , \alpha_{g_r}^{1,\,0}\}$ be a $\C$-basis of $E_r^{1,\,0}(X)$. Thanks to (ii) of Lemma \ref{Lem:E_r_1-0_singleton_ker}, the forms $\alpha_1^{1,\,0},\dots , \alpha_{g_r}^{1,\,0}$ are also linearly independent in $H_{\bar\partial}^{1,\,0}(X,\,\C) = \C\,\langle\omega^1,\dots , \omega^n\rangle$. Thus, there exist $c_{jl}\in\C$ such that \begin{eqnarray*}\alpha_j^{1,\,0} = \sum\limits_{l=1}^n c_{jl}\,\omega^l,  \hspace{5ex} j=1,\dots , g_r,\end{eqnarray*} and the coefficient vectors $\tilde{c}_j:=(c_{j1},\dots , c_{jn})\in\C^n$, with $j=1,\dots , g_r$, are necessarily linearly independent. Coupled with the fact that, for every $x\in X$, $\{\omega^1(x),\dots , \omega^n(x)\}$ is a $\C$-basis of $\Lambda^{1,\,0}T^\star_xX$, this implies that the coefficient matrix of $\alpha_1^{1,\,0}(x),\dots , \alpha_{g_r}^{1,\,0}(x)$ with respect to this basis has maximal rank (equal to $g_r$). We infer that $\alpha_1^{1,\,0},\dots , \alpha_{g_r}^{1,\,0}$ are pointwise linearly independent.  

Now, recall that $d\alpha_j^{1,\,0}=0$ on $X$, for every $j$. Fix a point $x\in X$ and a small simply connected open subset $U$ of $X$ containing $x$ such that $\alpha_j^{1,\,0} = df_j$ on $U$, for every $j=1,\dots , g_r$. It follows that each function $f_j$ is holomorphic on $U$ and, as we saw in the proof of Proposition \ref{Prop:E_r-Albanese-map_hol}, that we have \begin{eqnarray*}A_r(z) = \bigg(f_1(z) - f_1(x_0) + C_1,\dots , f_{g_r}(z) - f_{g_r}(x_0) + C_{g_r}\bigg) \hspace{3ex} \mbox{mod}\hspace{1ex} \Lambda, \hspace{5ex} z\in U,\end{eqnarray*} where $\Lambda$ is the group of periods of the $1$-forms $\alpha_1^{1,\,0},\dots , \alpha_{g_r}^{1,\,0}$.

Thus, the differential of $A_r$ at every $z\in U$ is \begin{eqnarray*}(dA_r)(z) = (\partial A_r)(z) = \bigg(\frac{\partial f_j}{\partial z_k}(z)\bigg)_{\substack{1\leq j\leq g_r\\1\leq k\leq n}},\end{eqnarray*} which is a matrix whose rows are the vectors $\alpha_1^{1,\,0},\dots , \alpha_{g_r}^{1,\,0}$. Furthermore, we have seen above that this matrix is of maximal rank at every point $z\in U$. We conclude that the $E_r$-Albanese map $A_r$ is submersive at every point $z\in X$.

It remains to show that $A_r$ is surjective. Since it is a submersion, $A_r:X\longrightarrow Alb_r(X)$ is an {\it open} map. On the other hand, $X$ is compact and $Alb_r(X)$ is Hausdorff, so $A_r$ is also a {\it closed} map. In particular, $A_r(X)$ is both open and closed in the complex torus $Alb_r(X)$. By connectedness of $Alb_r(X)$, we must have $A_r(X) = Alb_r(X)$, which proves that $A_r$ is surjective. \hfill $\Box$

\vspace{2ex}

Recall that the {\it algebraic dimension} $a(X)$ of a compact complex manifold $X$ is defined to be the largest number of algebraically independent meromorphic functions on $X$. It is well known that it never exceeds the dimension of $X$ as a complex manifold: $a(X)\leq n:=\mbox{dim}_\C X$. Moreover, it is maximal (i.e. $a(X)= n$) if and only if $X$ is {\it Moishezon} (i.e. $X$ is bimeromorphically equivalent to a projective manifold). Furthermore, it is known that $X$ is Moishezon if and only if there exists a {\it big} holomorphic line bundle $L$ on $X$. 

\begin{Cor}\label{Cor:A_r_c-par-page-ddbar_alg-dim_ineq} Let $X$ be a compact {\bf complex parallelisable} {\bf page-(r-1)-$\partial\bar\partial$-manifold} and let $A_r:X\longrightarrow Alb_r(X)$ be its $E_r$-Albanese map. Then, the algebraic dimensions compare as: \begin{eqnarray*}a\bigg(Alb_r(X)\bigg)\leq a(X).\end{eqnarray*}

\end{Cor}

\noindent {\it Proof.} By Corollary \ref{Cor:A_r_surjective-submersion_c-par-page-ddbar}, the map $A_r:X\longrightarrow Alb_r(X)$ is a holomorphic surjective submersion. Thus, if $f$ is a meromorphic function on $Alb_r(X)$, its lift $f\circ A_r$ is a meromorphic function on $X$.

Moreover, the algebraic independence of meromorphic functions is preserved by this lifting. Indeed, let $f_1,\dots , f_k$ be algebraically independent meromorphic functions on $Alb_r(X)$. If their lifts $f_1\circ A_r,\dots , f_k\circ A_r$ were not algebraically independent, there would exist a non-zero polynomial $P\in\C[T_1,\dots , T_k]$ such that $P(f_1\circ A_r,\dots , f_k\circ A_r) = 0$ as a meromorphic function on $X$. Now, $A_r$ is surjective, so, in particular, every point in the dense open subset of $Alb_r(X)$ where all the functions $f_1,\dots , f_k$ are holomorphic has a preimage under $A_r$. This implies that $P(f_1,\dots , f_k) = 0$ as a meromorphic function on $Alb_r(X)$, which contradicts the algebraic independence of $f_1,\dots , f_k$.

The contention follows. Note that we have not used the submersivity of the $E_r$-Albanese map $A_r:X\longrightarrow Alb_r(X)$, but only its holomorphicity and its surjectivity.   \hfill $\Box$

\vspace{2ex}

As an example of application of Corollary \ref{Cor:A_r_c-par-page-ddbar_alg-dim_ineq}, we can recover the following well-known fact:

\begin{Cor}\label{Cor:Iwasawa_alg-dim_2} The algebraic dimension of the $3$-dimensional Iwasawa manifold $I^{(3)}$ is $2$: \begin{eqnarray}\label{eqn:alg-dim_Iwasawa}a(I^{(3)}) = 2.\end{eqnarray}

\end{Cor}

\noindent {\it Proof.} The Iwasawa manifold is a {\it complex parallelisable page-$1$-$\partial\bar\partial$}-manifold (see [PSU20a, Proposition 4.3] for the latter claim, the former being standard) and our $E_2$-Albanese map $A_2:I^{(3)}\longrightarrow Alb_2(I^{(3)})$ coincides with the standard fibration of $I^{(3)}$ over an Abelian surface. In particular, the $E_2$-Albanese torus $Alb_2(I^{(3)})$ of $I^{(3)}$ is a $2$-dimensional complex projective manifold. Hence, it is also Moishezon, so $a(Alb_2(I^{(3)})) = 2$.

From our Corollary \ref{Cor:A_r_c-par-page-ddbar_alg-dim_ineq}, we get the first inequality below: \begin{eqnarray*}2=a\bigg(Alb_2(I^{(3)})\bigg)\leq a(I^{(3)})\leq 3,\end{eqnarray*} so $a(I^{(3)})\in\{2,\,3\}$. However, $a(I^{(3)})\neq 3$ because otherwise the Iwasawa manifold $I^{(3)}$ would be Moishezon. It would then also be a $\partial\bar\partial$-manifold, which is known not to be true. (Even worse, the Fr\"olicher spectral sequence of $I^{(3)}$ does not degenerate at $E1$, it only degenerates at $E_2$.) The contention follows. \hfill $\Box$

\section{Case of $E_r$-sGG manifolds}\label{section:E_r-sGG}  Let $X$ be a compact complex manifold with $\mbox{dim}_\C X =n$. We will identify a Hermitian metric on $X$ with the associated $C^\infty$ positive definite $(1,\,1)$-form on $X$. Recall that a Hermitian metric $\omega$ on $X$ is said to be a {\bf Gauduchon metric} (cf. [Gau77]) if $\partial\bar\partial\omega^{n-1} = 0$.

\subsection{Background}\label{subsection:background_sGG}

We start by recalling the following notions and the works where they were introduced:

\begin{Def}\label{Def:sG_definitions} (i)\, ([Pop13, Definition 4.1])\, A Hermitian metric $\omega$ on $X$ is said to be {\bf strongly Gauduchon (sG)} if the $(n,\,n-1)$-form $\partial\omega^{n-1}$ is $\bar\partial$-exact.

  If $X$ carries a strongly Gauduchon metric, $X$ is said to be a {\bf strongly Gauduchon (sG) manifold}.

\vspace{1ex}

(ii)\, ([PU, Definition 1.2 and Lemma 1.3])\, One says that $X$ is an {\bf sGG manifold} if every Gauduchon metric $\omega$ on $X$ is strongly Gauduchon.

\vspace{1ex}

(iii)\, ([Pop19, Definition 3.1])\, For a given integer $r\geq 1$, a Gauduchon metric $\omega$ on $X$ is said to be an {\bf $E_r$-sG metric} if $\partial\omega^{n-1}$ is $E_r$-exact.

If $X$ carries an $E_r$-sG metric, $X$ is said to be an {\bf $E_r$-sG manifold}.

If every Gauduchon metric $\omega$ on $X$ is an $E_r$-sG metric, $X$ is said to be an {\bf $E_r$-sGG manifold}.
  
\end{Def}

\vspace{2ex}

Recall that $\partial\omega^{n-1}$ being {\bf $E_r$-exact} (a condition that, for bidegree reasons, is non-trivial only for $r=1, 2, 3$) means that there exist forms $\zeta_{r-2}\in C^\infty_{n-1,\,n-1}(X,\,\C)$ and $\xi_0\in C^\infty_{n,\,n-2}(X,\,\C)$ such that \begin{eqnarray}\label{eqn:E_r_exact_def_1}\partial\omega^{n-1} = \partial\zeta_{r-2} + \bar\partial\xi_0\end{eqnarray} and such that $\zeta_{r-2}$ satisfies a further $(r-1)$ conditions (so, there are at most two of them): there exists a $C^\infty$ form $v_{r-3}^{(r-2)}\in C^\infty_{n-2,\,n}(X,\,\C)$ such that \begin{eqnarray}\label{eqn:E_r_exact_def_2}\bar\partial\zeta_{r-2} = \partial v_{r-3}^{(r-2)} \hspace{3ex}\mbox{and}\hspace{3ex}  \bar\partial v_{r-3}^{(r-2)} = 0.\end{eqnarray} Note that the last condition is always trivially satisfied, for bidegree reasons. However, the notions of $E_2$-exactness and $E_3$-exactness in bidegree $(n,\,n-1)$ are different, in general. 

  Specifically, $\partial\omega^{n-1}$ is {\bf $E_1$-exact} if and only if $\partial\omega^{n-1}\in\mbox{Im}\,\bar\partial$, while $\partial\omega^{n-1}$ is {\bf $E_2$-exact} if and only if there exist forms $\zeta\in C^\infty_{n-1,\,n-1}(X,\,\C)$ and $\xi\in C^\infty_{n,\,n-2}(X,\,\C)$ such that \begin{eqnarray}\label{eqn:E_2_exact_def}\partial\omega^{n-1} = \partial\zeta +  \bar\partial\xi \hspace{3ex}\mbox{and}\hspace{3ex} \bar\partial\zeta = 0.\end{eqnarray} The form $\partial\omega^{n-1}$ is {\bf $E_3$-exact} if and only if the first condition in (\ref{eqn:E_2_exact_def}) is satisfied and $\bar\partial\zeta\in\mbox{Im}\,\partial$.

  Since $E_1$-exactness is the usual $\bar\partial$-exactness, the notions of {\it $sG$-metrics} and {\it $E_1$-sG metrics} coincide. Hence, the notions of {\it sGG manifold} and {\it $E_1$-sGG manifold} coincide.

 On the other hand, for any Hermitian metric $\omega$ on $X$ and every integer $r\geq 1$, one has the following implications: \begin{eqnarray}\label{eqn:E_r-sG_implications}\omega \hspace{1ex} \mbox{is} \hspace{1ex} E_1\mbox{-sG} \implies \omega \hspace{1ex} \mbox{is} \hspace{1ex} E_2\mbox{-sG} \implies \omega \hspace{1ex} \mbox{is} \hspace{1ex} E_3\mbox{-sG} \implies \omega \hspace{1ex} \mbox{is Gauduchon}.\end{eqnarray}

  Moreover, the notion of {\it $E_r$-sGG manifold} generalises that of {\it page-(r-1)-$\partial\bar\partial$-manifold} thanks to the following

\begin{Prop}([PSU20b, Proposition 5.2])\label{Prop:page_r-1_ddbar_E_r-sGG} For any $r\in\N^\star$, the following implication holds: \begin{eqnarray}\label{eqn:page_r-1_ddbar_E_r-sGG}X \hspace{1ex} \mbox{is a {\bf page-(r-1)-$\partial\bar\partial$-manifold}} \implies X \hspace{1ex} \mbox{is an {\bf $E_r$-sGG manifold}}.\end{eqnarray}

\end{Prop}

\subsection{Hodge theory of $E_r$-sGG manifolds}\label{subsection:Hodge_E_r-sGG}

{\it $E_1$-sGG manifolds} were introduced and investigated in detail in [PU]. In this subsection, we generalise to {\it $E_r$-sGG manifolds} their properties that will be of interest to us in subsequent parts of this paper.

\vspace{1ex}

We start with the following characterisation. It generalises Lemma 1.3 in [PU].

\begin{Lem}\label{Lem:E_r-sGG_map-characterisation} Let $X$ be a compact complex manifold with $\mbox{dim}_\C X=n$. Fix $r\in\N^\star$.

\vspace{1ex}

$(1)$\, Let $H_A^{n-1,\,n-1}(X,\,\C)$ be the Aeppli cohomology group of bidegree $(n-1,\,n-1)$ of $X$. The following map is well defined: \begin{eqnarray}\label{eqn::E_r-sGG_map}T:H_A^{n-1,\,n-1}(X,\,\C)\longrightarrow E_r^{n,\,n-1}(X,\,\C), \hspace{5ex} T([\Omega]_A) = \{\partial\Omega\}_{E_r}.\end{eqnarray}

\vspace{1ex}

$(2)$\, Let ${\cal G}_X\subset H_A^{n-1,\,n-1}(X,\,\R)$ be the Gauduchon cone of $X$. The following equivalence holds: \begin{eqnarray*}\label{eqn:T-vanishing_Gauduchon-cone}T\equiv 0 \iff T_{|{\cal G}_X}\equiv 0.\end{eqnarray*}

\vspace{1ex}

$(3)$\, For any integer $r\geq 1$, the following equivalences hold: \begin{eqnarray}\label{eqn:E_r-sGG_equivalences}\nonumber X \hspace{1ex}\mbox{is $E_r$-sGG} \iff T\equiv 0 & \iff & \mbox{for every}\hspace{1ex} \Gamma\in C^\infty_{n,\,n-1}(X,\,\C) \hspace{1ex} \mbox{such that}\hspace{1ex} d\Gamma = 0, \\
\nonumber  & & \mbox{the following implication holds:} \\
  & & \Gamma\in\mbox{Im}\,\partial \implies \Gamma \hspace{1ex}\mbox{is $E_r$-exact}.\end{eqnarray}
  
\end{Lem}  

\noindent {\it Proof.} $(1)$\, If $\Omega\in C^\infty_{n-1,\,n-1}(X,\,\C)$ represents an Aeppli cohomology class, $\partial\bar\partial\Omega=0$. Thus, $\partial\Omega$ is $\bar\partial$-closed of bidegree $(n,\,n-1)$ and $\partial(\partial\Omega) = 0$. This proves that $\partial\Omega$ is $E_r$-closed for every $r\geq 1$, so $\partial\Omega$ defines a class $\{\partial\Omega\}_{E_r}\in E_r^{n,\,n-1}(X,\,\C)$.

To prove the independence of this class of the choice of representative $\Omega\in C^\infty_{n-1,\,n-1}(X,\,\C)$ of its Aeppli class, suppose that $\Omega$ defines the zero class in $H_A^{n-1,\,n-1}(X,\,\C)$. This amounts to $\Omega = \partial u^{n-2,\,n-1} + \bar\partial v^{n-1,\,n-2}$ for some forms of the shown bidegrees. Then $\partial\Omega = \bar\partial(-\partial v^{n-1,\,n-2})$, so $\partial\Omega$ is $\bar\partial$-exact (equivalently, $E_1$-exact), hence also $E_r$-exact for every $r\geq 1$. This means that $\{\partial\Omega\}_{E_r} = 0\in E_r^{n,\,n-1}(X,\,\C)$.

\vspace{1ex}

$(2)$\, This equivalence is obvious because ${\cal G}_X$ is a non-empty open cone in the real Aeppli cohomology space $H_A^{n-1,\,n-1}(X,\,\R)$ and $T$ is linear.

\vspace{1ex}

$(3)$\, $\bullet$ {\it Proof of the first equivalence}

\vspace{1ex}

By definition of the Gauduchon cone, the elements of ${\cal G}_X$ are the Aeppli cohomology classes $[\omega^{n-1}]_A$ induced by all the Gauduchon metrics $\omega$ on $X$. For any such $\omega$, the condition that $\partial\omega^{n-1}$ be $E_r$-exact is equivalent to $T([\partial\omega^{n-1}]_A) = 0\in E_r^{n,\,n-1}(X,\,\C)$. Using the definition of what it means for $X$ to be an $E_r$-sGG manifold (cf. (iii) of Definition \ref{Def:sG_definitions}), we deduce the equivalence: \begin{eqnarray*}X \hspace{1ex}\mbox{is $E_r$-sGG} \iff T_{|{\cal G}_X}\equiv 0.\end{eqnarray*} Combined with the equivalence under $(2)$, this proves the first equivalence under $(3)$.

\vspace{1ex}

$\bullet$ {\it Proof of the second equivalence}

\vspace{1ex}

``$\implies$'' Suppose $T\equiv 0$. Let $\Gamma\in C^\infty_{n,\,n-1}(X,\,\C)$ be such that $d\Gamma=0$ and $\Gamma = \partial\Omega$ for some smooth $(n-1,\,n-1)$-form $\Omega$. We get: \begin{eqnarray*}0 = d\Gamma = \bar\partial\partial\Omega.\end{eqnarray*} This shows that $\Omega$ defines an Aeppli class $[\Omega]_A\in H_A^{n-1,\,n-1}(X,\,\C)$.

Now, the hypothesis $T\equiv 0$ yields the first equality in: $0 = T([\Omega]_A) = \{\partial\Omega\}_{E_r}$. We conclude that $\Gamma = \partial\Omega$ is $E_r$-exact, as desired.

\vspace{1ex}

``$\impliedby$'' Let $\Omega\in C^\infty_{n-1,\,n-1}(X,\,\C)$ such that $\partial\bar\partial\Omega = 0$. Then $\Gamma:=\partial\Omega\in C^\infty_{n,\,n-1}(X,\,\C)$ satisfies the hypotheses of the implication on the right-hand side of the second equivalence in $(3)$. Therefore, $\Gamma=\partial\Omega$ is $E_r$-exact, which yields $T([\Omega]_A) = \{\partial\Omega\}_{E_r} = 0$.

We conclude that $T\equiv 0$.   \hfill $\Box$

\vspace{2ex}

We are now in a position to prove that, in degrees $1$ and $2n-1$, $E_r$-sGG manifolds satisfy a weaker {\bf Hodge decomposition}-type property than page-(r-1)-$\partial\bar\partial$-manifolds. It is nearly as strong as the analogous property in those degrees of the latter class of manifolds (cf. Definition \ref{Def:page-1-dd-bar}) and it characterises the former class. The following result generalises Theorems 1.6 and 3.1 of [PU].

\begin{The}\label{The:E_r-sGG_Hodge-decomp} Let $X$ be a compact complex manifold with $\mbox{dim}_\C X=n$. Fix $r\in\N^\star$. 

\vspace{1ex}

$(1)$\, The following canonical $\C$-linear map is {\bf well defined} and {\bf injective}: \begin{eqnarray}\label{eqn:F-map_def}F_r:H^1_{dR}(X,\,\C)\longrightarrow E_r^{0,\,1}(X)\oplus\overline{E_r^{0,\,1}(X)}, \hspace{5ex} \{\alpha\}_{dR}\longmapsto\bigg(\{\alpha^{0,\,1}\}_{E_r},\,\overline{\{\overline{\alpha^{1,\,0}}}\}_{E_r}\bigg).\end{eqnarray}

Hence, $b_1(X)\leq 2\,e_r^{0,\,1}(X)$, where $b_1(X)=\mbox{dim}_\C H^1_{dR}(X,\,\C)$ and $e_r^{0,\,1}(X)=\mbox{dim}_\C E_r^{0,\,1}(X)$.

\vspace{1ex}

$(2)$\, The following canonical $\C$-linear map is {\bf well defined}: \begin{eqnarray}\label{eqn:F-star-map_def}F_r^\star:E_r^{n,\,n-1}(X)\oplus\overline{E_r^{n,\,n-1}(X)}\longrightarrow H^{2n-1}_{dR}(X,\,\C), \hspace{5ex} \bigg(\{\beta\}_{E_r},\,\overline{\{\gamma\}}_{E_r}\bigg)\longmapsto \{\beta + \overline\gamma\}_{dR}.\end{eqnarray} Moreover, the maps $F_r^\star$ and $F_r$ are {\bf dual} to each other. Hence, $F_r^\star$ is {\bf surjective}.

\vspace{1ex}

$(3)$\, The following equivalences hold: \begin{eqnarray*}X \hspace{1ex}\mbox{is {\bf $E_r$-sGG}} \iff F_r \hspace{1ex}\mbox{is {\bf bijective}} \iff F_r^\star \hspace{1ex}\mbox{is {\bf bijective}} \iff b_1(X) = 2\,e_r^{0,\,1}(X).\end{eqnarray*}

\end{The}  

\noindent {\it Proof.} $(1)$\, For any $1$-form $\alpha = \alpha^{1,\,0} + \alpha^{0,\,1}$ and $r\geq 2$, we have the equivalences: \begin{eqnarray*}d\alpha = 0 & \iff & \partial\alpha^{1,\,0} = 0 \hspace{2ex}\mbox{and}\hspace{2ex} \partial\alpha^{0,\,1} + \bar\partial\alpha^{1,\,0} = 0 \hspace{2ex}\mbox{and}\hspace{2ex} \bar\partial\alpha^{0,\,1} = 0 \\
  & \iff & \bigg(\bar\partial\alpha^{0,\,1} = 0 \hspace{2ex}\mbox{and}\hspace{2ex} \partial\alpha^{0,\,1} = \bar\partial(-\alpha^{1,\,0}) \hspace{2ex}\mbox{and}\hspace{2ex} \partial(-\alpha^{1,\,0}) = 0\bigg), \\
  & & \bigg(\bar\partial\overline{\alpha^{1,\,0}} = 0 \hspace{2ex}\mbox{and}\hspace{2ex} \partial\overline{\alpha^{1,\,0}} = \bar\partial(-\overline{\alpha^{0,\,1}}) \hspace{2ex}\mbox{and}\hspace{2ex} \partial(-\overline{\alpha^{0,\,1}}) = 0\bigg) \\
  & \iff & \alpha^{0,\,1} \hspace{1ex}\mbox{is}\hspace{1ex} E_r\mbox{-closed} \hspace{3ex}\mbox{and}\hspace{3ex} \overline{\alpha^{1,\,0}} \hspace{1ex}\mbox{is}\hspace{1ex} E_r\mbox{-closed}.\end{eqnarray*} If $r=1$, the second and third equivalences above are only implications. In all cases, for every $1$-form $\alpha$, the classes $\{\alpha^{0,\,1}\}_{E_r}\in E_r^{0,\,1}(X)$ and $\{\overline{\alpha^{1,\,0}}\}_{E_r}\in E_r^{0,\,1}(X)$ are well defined.

To show that $F_r$ is independent of the choice of representative of the class $\{\alpha\}_{dR}\in H^1_{dR}(X,\,\C)$, we let $\alpha = \alpha^{1,\,0} + \alpha^{0,\,1} = df$ for some smooth function $f$ on $X$. This is equivalent to: \begin{eqnarray*}\alpha^{1,\,0} = \partial f  \hspace{2ex}\mbox{and}\hspace{2ex} \alpha^{0,\,1} = \bar\partial f & \implies & \overline{\alpha^{1,\,0}} \hspace{2ex}\mbox{and}\hspace{2ex} \alpha^{0,\,1} \hspace{1ex}\mbox{are}\hspace{1ex} E_1\mbox{-exact} \implies \\
  \overline{\alpha^{1,\,0}} \hspace{2ex}\mbox{and}\hspace{2ex} \alpha^{0,\,1} \hspace{1ex}\mbox{are}\hspace{1ex} E_r\mbox{-exact for every $r\geq 1$}  & \iff & \{\alpha^{0,\,1}\}_{E_r} = \overline{\{\overline{\alpha^{1,\,0}}}\}_{E_r}  =0 \hspace{1ex} \mbox{for every $r\geq 1$}.\end{eqnarray*} This finishes the proof of the well-definedness of $F$.

\vspace{1ex}

To prove that $F_r$ is injective, let $\alpha = \alpha^{1,\,0} + \alpha^{0,\,1}$ be a $1$-form on $X$ such that $d\alpha = 0$ and $\alpha^{0,\,1}$ and $\overline{\alpha^{1,\,0}}$ are $E_r$-exact. Since in bidegree $(0,\,1)$, $E_r$-exactness is equivalent to $\bar\partial$-exactness, there exist smooth functions $f,g$ on $X$ such that $\alpha^{0,\,1} = \bar\partial f$ and $\overline{\alpha^{1,\,0}} = \bar\partial g$. This implies that \begin{eqnarray*}\alpha = \alpha^{1,\,0} + \alpha^{0,\,1} = \partial\bar{g} + \bar\partial f.\end{eqnarray*} Since $d\alpha = 0$, we get:  \begin{eqnarray*}0 = \partial\alpha^{0,\,1} + \bar\partial\alpha^{1,\,0} = \partial\bar\partial(f-\bar{g})  \hspace{5ex} \mbox{on}\hspace{1ex} X.\end{eqnarray*} Since $X$ is compact, the function $f-\bar{g}$ must constant. We get $\bar{g} = f + C$, hence $\alpha = \partial(f+C) + \bar\partial f = df$. This is tantamount to $\{\alpha\}_{dR} = 0$.

\vspace{1ex}

$(2)$\, Every $(n,\,n-1)$-form is trivially $\partial$-closed, so in bidegree $(n,\,n-1)$, $\bar\partial$-closedness (i.e. $E_1$-closedness) is equivalent to $E_r$-closedness for every $r\geq 1$. Thus, for any $E_r$-closed $(n,\,n-1)$-forms $\beta, \gamma$, we have $d\beta = d\gamma = 0$, so $\beta + \bar{\gamma}$ is $d$-closed and therefore it defines a De Rham class in $H^{2n-1}_{dR}(X,\,\C)$. To show that $F_r^{\star}$ is independent of the choices of representatives of the classes $\{\beta\}_{E_r}, \{\gamma\}_{E_r}\in E_r^{n,\,n-1}(X)$, suppose that $\beta$ and $\gamma$ are $E_r$-exact. We have seen in (\ref{eqn:E_r_exact_def_1}) and (\ref{eqn:E_r_exact_def_2}) that in bidegree $(n,\,n-1)$, only for $r=2, 3$ can $E_r$-exactness be strictly weaker than for the previous value of $r$. Since the case $r=1$ was treated in [PU, Theorem 3.1, (ii)], we will only deal with the cases $r=2$ and $r=3$.

If $\beta\in C^\infty_{n,\,n-1}(X,\,\C)$ is $E_2$-exact, then, as in (\ref{eqn:E_2_exact_def}), there exist forms $\zeta\in C^\infty_{n-1,\,n-1}(X,\,\C)$ and $\xi\in C^\infty_{n,\,n-2}(X,\,\C)$ such that \begin{eqnarray*}\beta = \partial\zeta +  \bar\partial\xi \hspace{3ex}\mbox{and}\hspace{3ex} \bar\partial\zeta = 0.\end{eqnarray*} This implies that $\beta = d(\zeta + \xi)$, so $\beta$ is $d$-exact. Similarly, if $\gamma\in C^\infty_{n,\,n-1}(X,\,\C)$ is $E_2$-exact, then $\gamma$ is $d$-exact. We get: $\{\beta + \overline\gamma\}_{dR} = 0\in H^{2n-1}_{dR}(X,\,\C)$.

If $\beta\in C^\infty_{n,\,n-1}(X,\,\C)$ is $E_3$-exact, then there exist forms $\zeta\in C^\infty_{n-1,\,n-1}(X,\,\C)$, $\xi\in C^\infty_{n,\,n-2}(X,\,\C)$ and $\eta\in C^\infty_{n-2,\,n}(X,\,\C) $ such that \begin{eqnarray*}\beta = \partial\zeta +  \bar\partial\xi \hspace{3ex}\mbox{and}\hspace{3ex} \bar\partial\zeta = \partial\eta.\end{eqnarray*} This implies that $\beta = d(\zeta + \xi - \eta)$, so $\beta$ is $d$-exact. Similarly, if $\gamma\in C^\infty_{n,\,n-1}(X,\,\C)$ is $E_3$-exact, then $\gamma$ is $d$-exact. Again, we get: $\{\beta + \overline\gamma\}_{dR} = 0\in H^{2n-1}_{dR}(X,\,\C)$.

We conclude that $F_r^{\star}$ is well defined.

\vspace{1ex}

We now prove that the maps $F_r$ and $F_r^{\star}$ are dual to each other. Under the Serre-type duality $E^{n,\,n-1}_r(X)\simeq(E^{0,\,1}_r(X))^{\star}$ of Theorem \ref{The:duality_E_r}, every pair $(\{\beta\}_{E_r},\,\overline{\{\gamma\}}_{E_r})\in E^{n,\,n-1}_r(X)\oplus\overline{E^{n,\,n-1}_r(X)}$ can be identified with the pair $(u,\,\bar{v})$ in which $u, v\,:\,E^{0,\,1}_r(X)\rightarrow\C$ are the $\C$-linear maps acting as \begin{eqnarray*}u(\{\alpha^{0,\,1}\}_{E_r}) = \int\limits_X\beta\wedge\alpha^{0,\,1} \hspace{2ex} \mbox{and} \hspace{2ex} v(\{\alpha^{0,\,1}\}_{E_r}) = \int\limits_X\gamma\wedge\alpha^{0,\,1}\end{eqnarray*} for every class $\{\alpha^{0,\,1}\}_{E_r}\in E^{0,\,1}_r(X)$ and $\bar{v}\,:\,\overline{E^{0,\,1}_r(X)}\to\C$ is the $\C$-linear map defined by $\bar{v}(\overline{\{\alpha^{0,\,1}\}}_{E_r}):=\overline{v(\{\alpha^{0,\,1}\}_{E_r})}$. Proving the duality between $F_r$ and $F_r^{\star}$ amounts to proving that \begin{equation}\label{eqn:dualityFstar}\sigma_{\beta + \bar\gamma} = (u + \bar{v})\circ F\end{equation} for any $\{\beta\}_{E_r}, \{\gamma\}_{E_r}\in E^{n,\,n-1}_r(X,\,\C)$, where $\sigma_{\beta + \bar\gamma}\,:\,H^1_{DR}(X,\,\C)\rightarrow\C$ is the $\C$-linear map representing the de Rham class \begin{eqnarray*}F_r^{\star}(\{\beta\}_{E_r},\,\overline{\{\gamma\}}_{E_r}) = \{\beta + \bar{\gamma}\}_{dR}\in H^{2n-1}_{dR}(X,\,\C)\simeq(H^1_{dR}(X,\,\C))^{\star}\end{eqnarray*} under Poincar\'e duality. By definition, this means that for every $\{\alpha\}_{dR}\in H^1_{dR}(X,\,\C)$ we have: \begin{eqnarray*}\sigma_{\beta + \bar\gamma}(\{\alpha\}_{dR}) = \int\limits_X(\beta + \bar\gamma)\wedge\alpha = \int\limits_X\beta\wedge\alpha^{0,\,1} + \int\limits_X\bar\gamma\wedge\alpha^{1,\,0} = \bigg((u + \bar{v})\circ F_r\bigg)(\{\alpha\}_{dR}).\end{eqnarray*}

\noindent This proves (\ref{eqn:dualityFstar}). We conclude that $F_r$ and $F_r^{\star}$ are dual to each other.

\vspace{1ex}

$(3)$\, It suffices to prove the equivalence: \begin{eqnarray*}X \hspace{1ex}\mbox{is {\bf $E_r$-sGG}} \iff  F^\star \hspace{1ex}\mbox{is {\bf injective}}.\end{eqnarray*}

To prove the implication ``$\implies$'', suppose that $X$ is an $E_r$-sGG manifold. Let $\{\beta\}_{E_r}, \{\gamma\}_{E_r}\in E^{n,\,n-1}_r(X)$, with $\beta$ and $\gamma$ of bidegree $(n,\,n-1)$, such that $\beta + \bar{\gamma} = d(\Omega^{n,\,n-2} + \Omega^{n-1,\,n-1} + \Omega^{n-2,\,n})$ for some forms $\Omega^{p,\,q}$ of the specified bidegrees. This amounts to having \begin{eqnarray*}\beta - \bar\partial\Omega^{n,\,n-2} = \partial\Omega^{n-1,\,n-1} \hspace{2ex}\mbox{and}\hspace{2ex} \gamma - \bar\partial\overline{\Omega^{n-2,\,n}} = \partial\overline{\Omega^{n-1,\,n-1}}.\end{eqnarray*}

Now, because $\bar\partial$-exactness ($=E_1$-exactness) implies $E_r$-exactness for every $r\geq 1$, we have: \begin{eqnarray*}\beta - \bar\partial\Omega^{n,\,n-2}\in\{\beta\}_{E_r} \hspace{3ex}\mbox{and}\hspace{3ex} \gamma - \bar\partial\overline{\Omega^{n-2,\,n}}\in\{\gamma\}_{E_r}.\end{eqnarray*}

On the other hand, $\partial\Omega^{n-1,\,n-1}$ and $\partial\overline{\Omega^{n-1,\,n-1}}$ are $d$-closed (because $\bar\partial\beta = \bar\partial\gamma = 0$ since $\beta$ and $\gamma$ are both $E_r$-closed, as representatives of $E_r$-classes) and $\partial$-exact $(n,\,n-1)$-forms on $X$. Thus, the $E_r$-sGG assumption on $X$ implies (thanks to $(3)$ of Lemma \ref{Lem:E_r-sGG_map-characterisation}) that they are both $E_r$-exact, i.e. $\{\beta\}_{E_r} = \{\gamma\}_{E_r} = 0$ in $E^{n,\,n-1}_r(X)$. We conclude that $F_r^{\star}$ is injective if $X$ is $E_r$-sGG.

\vspace{1ex}

To prove the reverse implication ``$\Longleftarrow$'', suppose that $F_r^{\star}$ is injective. We will show that every Gauduchon metric on $X$ is $E_r$-sG. This will imply that $X$ is an $E_r$-sGG manifold thanks to $(iii)$ of Definition \ref{Def:sG_definitions}.

Let $\omega$ be any Gauduchon metric on $X$. Then $\partial\omega^{n-1}\in\ker\bar\partial$. Since $\partial(\partial\omega^{n-1})=0$ (trivially), the $(n,\,n-1)$-form $\partial\omega^{n-1}$ is $E_r$-closed for every $r\geq 1$. Therefore, we have a class $\{\partial\omega^{n-1}\}_{E_r}\in E^{n,\,n-1}_r(X)$.

Now, $\partial\omega^{n-1} + \overline{\partial\omega^{n-1}} = d\omega^{n-1}$, so \begin{eqnarray*}F_r^{\star}(\{\partial\omega^{n-1}\}_{E_r},\,\overline{\{\partial\omega^{n-1}\}_{E_r}}) = \{\partial\omega^{n-1} + \overline{\partial\omega^{n-1}}\}_{dR} = \{d\omega^{n-1}\}_{dR} = 0.\end{eqnarray*} Since $F_r^{\star}$ is supposed injective, we infer that $\{\partial\omega^{n-1}\}_{E_r} = 0$, i.e. $\omega$ is $E_r$-sG.  \hfill $\Box$

\vspace{2ex}

In the sequel, we will especially need the following

\begin{Cor}\label{Cor:E_r-sGG_Hodge-decomp} Let $X$ be a compact complex manifold with $\mbox{dim}_\C X=n$. Fix $r\in\N^\star$.

  The following equivalences hold: \begin{eqnarray*}X \hspace{1ex}\mbox{is {\bf $E_r$-sGG}} & \iff & X \hspace{1ex} \mbox{has a {\bf weak Hodge decomposition in degree $1$ at the $E_r$ level}}\\
    & & (\mbox{in the sense that the map $F_r$ of (\ref{eqn:F-map_def}) is {\bf bijective}}) \\
    & \iff & X \hspace{1ex} \mbox{has a {\bf weak Hodge decomposition in degree $2n-1$ at the $E_r$ level}}\\
  & & (\mbox{in the sense that the map $F_r^\star$ of (\ref{eqn:F-star-map_def}) is {\bf bijective}}). \end{eqnarray*}

\end{Cor}

\noindent {\it Proof.} It follows at once from Theorem \ref{The:E_r-sGG_Hodge-decomp}.  \hfill $\Box$

\vspace{2ex}

At the level of the dimensions, we get:

\begin{Cor}\label{Cor:E_r-sGG_b_1_2e_r-01} Let $X$ be a compact complex manifold with $\mbox{dim}_\C X=n$. Fix $r\in\N^\star$.

  The following inequalities hold: \begin{eqnarray*}b_1(X)\leq 2\,e_{r+1}^{0,\,1}(X)\leq 2\,e_r^{0,\,1}(X)\leq\dots\leq 2\,e_1^{0,\,1}(X).\end{eqnarray*}

  Moreover, the equality $b_1(X) = 2\,e_r^{0,\,1}(X)$ holds if and only if $X$ is an {\bf $E_r$-sGG manifold}. Consequently, we have the implications: \begin{eqnarray*}\label{eqn:E_r-sGG_implications}X \hspace{1ex} \mbox{is an {\bf $E_1$-sGG manifold}} \implies X \hspace{1ex} \mbox{is an {\bf $E_2$-sGG manifold}} \implies X \hspace{1ex} \mbox{is an {\bf $E_3$-sGG manifold}}.\end{eqnarray*}

\end{Cor}

\noindent {\it Proof.} It follows at once from Theorem \ref{The:E_r-sGG_Hodge-decomp}.  \hfill $\Box$

\vspace{2ex}

The last conclusion of Corollary \ref{Cor:E_r-sGG_b_1_2e_r-01} has been noticed before and also follows from the implications (\ref{eqn:E_r-sG_implications}) and (iii) of Definition \ref{Def:sG_definitions}. It is the analogue of the implication: \begin{eqnarray}\label{eqn:E_r-page-r_implications}X \hspace{1ex} \mbox{is a {\bf page-$(r-1)$-$\partial\bar\partial$}-manifold} \implies X \hspace{1ex} \mbox{is a {\bf page-$r$-$\partial\bar\partial$}-manifold}, \hspace{5ex} r\geq 1,\end{eqnarray} that was observed in [PSU20b, Introduction III].

\vspace{2ex}

The following consequence of the {\it weak Hodge decomposition in degree $1$ at the $E_r$ level} of $E_r$-sGG manifolds will come in handy later on.

\begin{Cor}\label{Cor:d-closed-rep_E_r_0-1} Let $X$ be a compact complex manifold. Suppose $X$ is {\bf $E_r$-sGG} for some $r\in\N^\star$.

  Then, every class $\mathfrak{c}_r^{0,\,1}\in E_r^{0,\,1}(X)$ can be represented by a unique $C^\infty$ {\bf $d$-closed} $(0,\,1)$-form $\widetilde\alpha^{0,\,1}$.

\end{Cor}

\noindent {\it Proof.} Let $\mathfrak{c}_r^{0,\,1}\in E_r^{0,\,1}(X)$. By the $E_r$-sGG assumption on $X$, the map $F_r:H^1_{dR}(X,\,\C)\longrightarrow E_r^{0,\,1}(X)\oplus\overline{E_r^{0,\,1}(X)}$ of (\ref{eqn:F-map_def}) is surjective. In particular, there exists a $C^\infty$ $1$-form $\alpha = \alpha^{1,\,0} + \alpha^{0,\,1}$ such that $d\alpha = 0$ and \begin{eqnarray*}F_r(\{\alpha\}_{dR}) = \bigg(\{\alpha^{0,\,1}\}_{E_r},\,\overline{\{\overline{\alpha^{1,\,0}}}\}_{E_r}\bigg) = \bigg(0,\,\overline{\mathfrak{c}_r^{0,\,1}}\bigg).\end{eqnarray*}

The condition $\{\alpha^{0,\,1}\}_{E_r} = 0\in E_r^{0,\,1}(X)$ means that $\alpha^{0,\,1}$ is $E_r$-exact. For bidegree reasons, this amounts to the existence of a $C^\infty$ function $g$ on $X$ such that $\alpha^{0,\,1} = \bar\partial g$.

Meanwhile, we get: $\mathfrak{c}_r^{0,\,1} = \{\overline{\alpha^{1,\,0}}\}_{E_r}$.

This means that the condition $d\alpha = 0$ reduces in this case to \begin{eqnarray*}\partial\alpha^{1,\,0} = 0 \hspace{3ex}\mbox{and}\hspace{3ex} \partial\bar\partial g + \bar\partial\alpha^{1,\,0} = 0.\end{eqnarray*}

Thus, if we set $\widetilde\alpha^{1,\,0}:= \alpha^{1,\,0} - \partial g\in C^\infty_{1,\,0}(X,\,\C)$, we have:  \begin{eqnarray*}d\widetilde\alpha^{1,\,0} = 0 \hspace{3ex}\mbox{and}\hspace{3ex} \bigg\{\overline{\widetilde\alpha^{1,\,0}}\bigg\}_{E_r} = \bigg\{\overline{\alpha^{1,\,0}}\bigg\}_{E_r} = \mathfrak{c}_r^{0,\,1}.\end{eqnarray*} The last property follows from $\overline{\widetilde\alpha^{1,\,0}} - \overline{\alpha^{1,\,0}} = -\bar\partial\bar{g}\in\mbox{Im}\,\bar\partial$, which implies that $\overline{\widetilde\alpha^{1,\,0}} - \overline{\alpha^{1,\,0}}$ is $E_s$-exact for every $s\geq 1$, hence also for $s=r$.

We see that $\widetilde\alpha^{0,\,1}:=\overline{\widetilde\alpha^{1,\,0}}$ is a $d$-closed $C^\infty$ representative of bidegree $(0,\,1)$ of the original class $\mathfrak{c}_r^{0,\,1}\in E_r^{0,\,1}(X)$. This proves the existence part.

To prove uniqueness, let $\widetilde\alpha^{0,\,1}\in C^\infty_{0,\,1}(X,\,\C)$ such that $d\widetilde\alpha^{0,\,1} = 0$ and $\widetilde\alpha^{0,\,1}$ is $E_r$-exact. For bidegree reasons, this last condition is equivalent to $\widetilde\alpha^{0,\,1} = \bar\partial g$ for some $C^\infty$ function $g$ on $X$. We get: \begin{eqnarray*}0 = \partial\widetilde\alpha^{0,\,1} = \partial\bar\partial g \hspace{5ex} \mbox{on}\hspace{2ex} X.\end{eqnarray*} Since $X$ is compact, $g$ must be constant and we get $\widetilde\alpha^{0,\,1} = \bar\partial g = 0$.  \hfill $\Box$

\subsection{The $E_r$-Jacobian and $E_r$-Albanese tori of $E_r$-sGG manifolds}\label{subsection:E_r-tori_E_r-sGG} We now use the Hodge theory derived for $E_r$-sGG manifolds in $\S$\ref{subsection:Hodge_E_r-sGG} to extend to this class of manifolds the $E_r$-Jacobian and $E_r$-Albanese tori, as well as the $E_r$-Albanese map, constructed for page-$(r-1)$-$\partial\bar\partial$-manifolds in $\S$\ref{subsection:E_r-Jacobi-torus_page_r-1},  $\S$\ref{subsection:E_r-Albanese-torus_page_r-1} and $\S$\ref{subsection:E_r-Albanese-map_page_r-1}. This will implicitly generalise to $E_r$-sGG manifolds the analogous constructions given in [Pop18] for sGG-manifolds.

\vspace{1ex}

Let $X$ be an $n$-dimensional compact complex $E_r$-sGG manifold. We have canonical maps: \begin{eqnarray*}H^1_{dR}(X,\,\R)\xhookrightarrow{j_1}H^1_{dR}(X,\,\C)\xrightarrow[\simeq]{F_r}E_r^{0,\,1}(X)\oplus\overline{E_r^{0,\,1}(X)}\xtwoheadrightarrow{p_r^{0,\,1}}E_r^{0,\,1}(X)\end{eqnarray*} and \begin{eqnarray*}H^{2n-1}_{dR}(X,\,\R)\xhookrightarrow{j_{2n-1}}H^{2n-1}_{dR}(X,\,\C)\xrightarrow[\simeq]{(F_r^\star)^{-1}}E_r^{n,\,n-1}(X)\oplus\overline{E_r^{n,\,n-1}(X)}\xtwoheadrightarrow{p_r^{n-1,\,n}}\overline{E_r^{n,\,n-1}(X)}.\end{eqnarray*} Moreover, the compositions: \begin{eqnarray*}p_r^{0,\,1}\circ F_r\circ j_1 : H^1_{dR}(X,\,\R)\longrightarrow E_r^{0,\,1}(X) \hspace{2ex}\mbox{and}\hspace{2ex} p_r^{n-1,\,n}\circ(F_r^\star)^{-1}\circ j_{2n-1} : H^{2n-1}_{dR}(X,\,\R)\longrightarrow\overline{E_r^{n,\,n-1}(X)}\end{eqnarray*} are {\it isomorphisms} for the same reasons as their counterparts of $\S$\ref{subsection:E_r-Jacobi-torus_page_r-1},  $\S$\ref{subsection:E_r-Albanese-torus_page_r-1} are. Finally, since the images of the canonical maps:  \begin{eqnarray*}J_1:H^1(X,\,\Z)\longrightarrow H^1_{dR}(X,\,\R)   \hspace{2ex}\mbox{and}\hspace{2ex} J_{2n-1}:H^{2n-1}(X,\,\Z)\longrightarrow H^{2n-1}_{dR}(X,\,\R)\end{eqnarray*} are {\it lattices} for any compact complex $n$-dimensional manifold $X$, we get the following

\begin{Cor}\label{Cor:lattice_E_r_sGG} If $X$ is a compact complex {\bf $E_r$-sGG manifold}, the images of the maps \begin{eqnarray*}p_r^{0,\,1}\circ F_r\circ j_1\circ J_1 : H^1(X,\,\Z) & \longrightarrow & E_r^{0,\,1}(X),\\
  p_r^{n-1,\,n}\circ (F_r^\star)^{-1}\circ j_{2n-1}\circ J_{2n-1} : H^{2n-1}(X,\,\Z) &\longrightarrow & \overline{E_r^{n,\,n-1}(X)},\end{eqnarray*} that we denote respectively by $\mbox{Im}^{(r)}H^1(X,\,\Z)$ and $\mbox{Im}^{(r)}H^{2n-1}(X,\,\Z)$, are {\bf lattices} in $E_r^{0,\,1}(X)$, respectively $\overline{E_r^{n,\,n-1}(X)}$.
 
\end{Cor}

\vspace{2ex}

On the other hand, the duality of Theorem \ref{The:duality_E_r} for the $E_r$ cohomology holds on any compact complex manifold $X$. In particular, we get a canonical isomorphism:  \begin{eqnarray*}E_r^{n,\,n-1}(X)\simeq E_r^{0,\,1}(X)^\star.\end{eqnarray*}

This enables us to introduce the following (cf. Definitions \ref{Def:E_r-Jacobi-torus} and \ref{Def:E_r-Albanese-torus}):

\begin{Def}\label{Def:E_r-Jacobi-Albanese-tori_E_r-sGG} Let $r\in\N^\star$ and let $X$ be a compact complex {\bf $E_r$-sGG manifold} with $\mbox{dim}_\C X = n$. Set $g_r:=\mbox{dim}_\C E_r^{0,\,1}(X) = (1/2)\,b_1(X)$. The $g_r$-dimensional complex torus \begin{eqnarray*}Jac_r(X):= E_r^{0,\,1}(X)\Big/\mbox{Im}^{(r)}H^1(X,\,\Z)\end{eqnarray*} is called the {\bf $E_r$-Jacobian torus} of $X$, while the $g_r$-dimensional complex torus \begin{eqnarray*}Alb_r(X):= \overline{E_r^{n,\,n-1}(X)}\Big/\mbox{Im}^{(r)}H^{2n-1}(X,\,\Z)\end{eqnarray*} is called the {\bf $E_r$-Albanese torus} of $X$.

\end{Def}

\vspace{2ex}

We now give the analogue in the more general $E_r$-sGG context of (iii) of Remark \ref{Rem:g_r_tori-duality}. It is based on the replacement of the existence of a $d$-closed $(p,\,q)$-type representative of every $E_r^{p,\,q}$-cohomology class (one of the key features of {\it page-(r-1)-$\partial\bar\partial$-manifolds}) by the existence of a {\bf canonical injection} \begin{eqnarray*}G_r:\overline{E_r^{0,\,1}(X)}\longrightarrow H^1_{dR}(X,\,\C)\end{eqnarray*} made possible by the {\it $E_r$-sGG assumption}.

\begin{Rem}\label{Rem:E_r-Albanese-torus_sGG_re} In the setting of Definition \ref{Def:E_r-Jacobi-Albanese-tori_E_r-sGG}, the {\bf $E_r$-Albanese torus} can be re-interpreted as \begin{eqnarray*}Alb_r(X)=\bigg(\overline{E_r^{0,\,1}(X)}\bigg)^\star\Big/\mbox{Im}^{(r)}H_1(X,\,\Z),\end{eqnarray*} where $\mbox{Im}^{(r)}H_1(X,\,\Z)$ is the image of the first homology group $H_1(X,\,\Z)$ under the canonical map: \begin{eqnarray}\label{eqn:map_H_1_E_r_0-1_bar_star_E_r-sGG}H_1(X,\,\Z)\longrightarrow\bigg(\overline{E_r^{0,\,1}(X)}\bigg)^\star, \hspace{5ex} [\xi]\longmapsto\widetilde{I}_\xi:=\bigg(\overline{E_r^{0,\,1}(X)}\ni\overline{\{\alpha^{0,\,1}\}_{E_r}}\longmapsto\int_\xi\overline{\{\alpha^{0,\,1}\}_{dR}}\in\C\bigg).\end{eqnarray}

  By $\overline{\{\alpha^{0,\,1}\}_{dR}}\in H^1_{dR}(X,\,C)$ we mean the image of the class $\overline{\{\alpha^{0,\,1}\}_{E_r}}\in\overline{E_r^{0,\,1}(X)}$ under the composition $G_r:=F_r^{-1}\circ\overline{j_r^{0,\,1}}$ of the following canonical linear maps: \begin{eqnarray}\label{eqn:composition_F_r_inverse}\overline{E_r^{0,\,1}(X)}\xhookrightarrow{\overline{j_r^{0,\,1}}} E_r^{0,\,1}(X)\oplus\overline{E_r^{0,\,1}(X)}\xrightarrow[\simeq]{F_r^{-1}}H^1_{dR}(X,\,\C),\end{eqnarray} where $\overline{j_r^{0,\,1}}$ is the canonical injection of a vector space into its direct sum with another and $F_r^{-1}$ is the inverse of the {\bf weak Hodge decomposition} isomorphism of (\ref{eqn:F-map_def}) that characterises {\bf $E_r$-sGG manifolds}. As in (iii) of Remark \ref{Rem:g_r_tori-duality}, the last integral in (\ref{eqn:map_H_1_E_r_0-1_bar_star_E_r-sGG}) is well defined once a homology class $[\xi]\in H_1(X,\,\Z)$ has been fixed since the current of integration $I_\xi$ on a loop $\xi$  is $d$-closed.

\end{Rem}

\subsection{The $E_r$-Albanese map of an $E_r$-sGG manifold}\label{subsection:E_r-Albanese-map_E_r-sGG} The definitions and results in this subsection are very similar to their {\it page-(r-1)-$\partial\bar\partial$}-counterparts of $\S$\ref{subsection:E_r-Albanese-map_page_r-1}. We will only point out the differences.

\vspace{1ex}

The only difference between the next definition and its {\it page-(r-1)-$\partial\bar\partial$}-counterpart of Definition \ref{Def:E_r-Albanese-map} is, once again, the replacement of a $d$-closed $(0,\,1)$-representative of every class in $E_r^{0,\,1}(X)$ (which need not exist for $E_r$-sGG manifolds) by the canonical image of $\overline{E_r^{0,\,1}(X)}$ in $H^1_{dR}(X,\,\C)$.

\begin{Def}\label{Def:E_r-Albanese-map_E_r-sGG} Let $r\in\N^\star$ and let $X$ be a compact complex {\bf $E_r$-sGG manifold} with $\mbox{dim}_\C X =n$. Fix a base point $a\in X$. For every point $x\in X$, let $\xi_x$ be a path in $X$ from $a$ to $x$ and let $\widetilde{I}_{\xi_x}\in\bigg(\overline{E_r^{0,\,1}(X)}\bigg)^\star$ be the linear functional defined by \begin{eqnarray}\label{eqn:map_H_1_E_r_0-1_bar_star_paths_E_r-sGG}\overline{E_r^{0,\,1}(X)}\ni\overline{\mathfrak{c}}\xmapsto{\widetilde{I}_{\xi_x}}\int_{\xi_x}\alpha({\overline{\mathfrak{c}}})\in\C,\end{eqnarray} where, for each class $\mathfrak{c}\in E_r^{0,\,1}(X)$, $\alpha({\overline{\mathfrak{c}}})$ denotes a $1$-form on $X$ representing the de Rham class $\overline{\mathfrak{c}}_{dR} = G_r(\overline{\mathfrak{c}})\in H^1_{dR}(X,\,\C)$, the image of $\overline{\mathfrak{c}}\in\overline{E_r^{0,\,1}(X)}$ under the canonical injection of (\ref{eqn:composition_F_r_inverse}): \begin{eqnarray*}G_r=F_r^{-1}\circ\overline{j_r^{0,\,1}}:\overline{E_r^{0,\,1}(X)}\longrightarrow H^1_{dR}(X,\,\C).\end{eqnarray*}
  
 The canonical map: \begin{eqnarray}\label{eqn:E_r-Albanese-map_E_r-sGG}A_r:X\longrightarrow Alb_r(X)=\bigg(\overline{E_r^{0,\,1}(X)}\bigg)^\star\Big/\mbox{Im}^{(r)}H_1(X,\,\Z), \hspace{5ex} x\longmapsto \widetilde{I}_{\xi_x} \hspace{3ex} \mbox{mod}\hspace{1ex} \mbox{Im}^{(r)}H_1(X,\,\Z),\end{eqnarray} is called the {\bf $E_r$-Albanese map} of $X$.

\end{Def}

\vspace{2ex}

As with Definition \ref{Def:E_r-Albanese-map}, the integral \begin{eqnarray*}\int_{\xi_x}\alpha({\overline{\mathfrak{c}}}) = \int\limits_X I_{\xi_x}\wedge\alpha({\overline{\mathfrak{c}}})\end{eqnarray*} depends on the choice of path $\xi_x$ from $a$ to $x$ despite the fact that $\alpha({\overline{\mathfrak{c}}})$ is necessarily $d$-closed (as a representative of a de Rham class). The integral also depends on the choice of this representative. Nevertheless, the $E_r$-Albanese map $A_r:X\longrightarrow Alb_r(X)$ is well defined for the same reason as in Definition \ref{Def:E_r-Albanese-map}: the ambiguity in the definition of the above integral is captured by $H_1(X,\,\Z)$, whose image is quotiented out.

\vspace{2ex}

Next, we point out that Proposition \ref{Prop:E_r-Albanese-map_hol}, too, generalises to the $E_r$-sGG case with only one modification: the vector space $E_r^{1,\,0}(X)$ has to be replaced by \begin{eqnarray}\label{eqn:cal-E_r-1_def}{\cal E}^1_r(X):=G_r\bigg(\overline{E_r^{0,\,1}(X)}\bigg)\subset H^1_{dR}(X,\,\C),\end{eqnarray} where $G_r=F_r^{-1}\circ\overline{j_r^{0,\,1}}:\overline{E_r^{0,\,1}(X)}\longrightarrow H^1_{dR}(X,\,\C)$ is the canonical injection of (\ref{eqn:composition_F_r_inverse}).

In particular, whenever $X$ is an {\bf $E_r$-sGG manifold}, we have a {\bf canonical isomorphism}: \begin{eqnarray}\label{eqn:cal-E_r-1_can-isom}\overline{E_r^{0,\,1}(X)}\simeq{\cal E}^1_r(X)\subset H^1_{dR}(X,\,\C).\end{eqnarray}

Corollary \ref{Cor:d-closed-rep_E_r_0-1} can be reformulated as

\begin{Cor}\label{Cor:cal-E_r-1_1-0_rep} Let $X$ be a compact complex manifold. Suppose $X$ is {\bf $E_r$-sGG} for some $r\in\N^\star$.

  Then, every de Rham class $\mathfrak{c}^{(r)}_{dR}\in{\cal E}^1_r(X)\subset H^1_{dR}(X,\,\C)$ can be represented by a unique $C^\infty$ (necessarily $d$-closed) $(1,\,0)$-form $\widehat{\alpha^{1,\,0}}$.

\end{Cor}

\noindent {\it Proof.} Every class in ${\cal E}^1_r(X)$ is of the form $\mathfrak{c}^{(r)}_{dR} = G_r(\overline{\mathfrak{c}_r^{0,\,1}})$ for a unique class $\mathfrak{c}_r^{0,\,1}\in E_r^{0,\,1}(X)$. Now, by Corollary \ref{Cor:d-closed-rep_E_r_0-1}, every class of the latter type can be represented by a $C^\infty$ {\bf $d$-closed} $(0,\,1)$-form $\widetilde\alpha^{0,\,1}$.

Thus, we can choose \begin{eqnarray*}\widehat{\alpha^{1,\,0}}:=\overline{\widetilde\alpha^{0,\,1}}.\end{eqnarray*} This proves existence.

To prove uniqueness, let $\widehat{\alpha^{1,\,0}}\in C^\infty_{1,\,0}(X,\,\C)$ be such that $\widehat{\alpha^{1,\,0}} = df$ for some $C^\infty$ function $f$ on $X$. Since $\widehat{\alpha^{1,\,0}}$ is of bidegree $(1,\,0)$, we get $\bar\partial f = 0$ on $X$. By compactness of $X$, $f$ must be constant, so $\widehat{\alpha^{1,\,0}} = \partial f = 0$. \hfill $\Box$

\vspace{2ex}

Conversely, we have the following

\begin{Cor}\label{Cor:cal-E_r-1_1-0_rep_converse} Let $X$ be a compact complex manifold. Suppose $X$ is {\bf $E_r$-sGG} for some $r\in\N^\star$.

Every $C^\infty$ $d$-closed $(1,\,0)$-form $\widehat{\alpha^{1,\,0}}$ on $X$ represents a de Rham class $\mathfrak{c}^{(r)}_{dR}\in{\cal E}^1_r(X)\subset H^1_{dR}(X,\,\C)$.

\end{Cor}

\noindent {\it Proof.} Let $\widehat{\alpha^{1,\,0}}$ be a form as in the statement. Then $\overline{\widehat{\alpha^{1,\,0}}}\in C^\infty_{0,\,1}(X,\,\C)$ and $d\overline{\widehat{\alpha^{1,\,0}}} = 0$. The last property means that $\overline{\widehat{\alpha^{1,\,0}}}$ is both $\partial$- and $\bar\partial$-closed, hence $E_r$-closed. Therefore, $\overline{\widehat{\alpha^{1,\,0}}}$ represents a class $\bigg\{\overline{\widehat{\alpha^{1,\,0}}}\bigg\}_{E_r}\in E_r^{0,\,1}(X)$. This implies that the original form $\widehat{\alpha^{1,\,0}}$ represents the class \begin{eqnarray*}\mathfrak{c}^{(r)}_{dR}:=G_r\bigg(\overline{\bigg\{\overline{\widehat{\alpha^{1,\,0}}}\bigg\}_{E_r}}\bigg)\in G_r\bigg(\overline{E_r^{0,\,1}(X)}\bigg) = {\cal E}^1_r(X)\subset H^1_{dR}(X,\,\C).\end{eqnarray*} \hfill $\Box$

\vspace{2ex}

The conclusion of the last two corollaries is the following

\begin{Cor}\label{Cor:cal-E-1-r_d-closed_1-0_canonical} Let $X$ be a compact complex manifold. Suppose $X$ is {\bf $E_r$-sGG} for some $r\in\N^\star$. 

\vspace{1ex}

(i)\, ${\cal E}^1_r(X)$ identifies canonically with the $\C$-vector space of $d$-closed $C^\infty$ $(1,\,0)$-forms on $X$. 

This identification maps each class $\mathfrak{c}^{(r)}_{dR}\in{\cal E}^1_r(X)$ to its unique $d$-closed $C^\infty$ $(1,\,0)$-representative.

\vspace{1ex}

(ii)\, The canonical linear map \begin{eqnarray*}{\cal E}^1_r(X)\ni\{\widehat{\alpha^{1,\,0}}\}_{dR}\longmapsto[\widehat{\alpha^{1,\,0}}]_{\bar\partial}\in H^{1,\,0}_{\bar\partial}(X,\,\C)\end{eqnarray*} is well defined and {\bf injective}.

\end{Cor}  

\noindent {\it Proof.} (i) follows at once from Corollaries \ref{Cor:cal-E_r-1_1-0_rep} and \ref{Cor:cal-E_r-1_1-0_rep_converse}.

\vspace{1ex}

(ii)\, If $\widehat{\alpha^{1,\,0}}$ is a smooth $(1,\,0)$-form such that $d\widehat{\alpha^{1,\,0}} = 0$, then $\partial\widehat{\alpha^{1,\,0}} = 0$ and $\bar\partial\widehat{\alpha^{1,\,0}} = 0$. In particular, $\widehat{\alpha^{1,\,0}}$ defines a Dolbeault class $[\widehat{\alpha^{1,\,0}}]_{\bar\partial}\in H^{1,\,0}_{\bar\partial}(X,\,\C)$. Moreover, by Corollary \ref{Cor:cal-E_r-1_1-0_rep}, $\widehat{\alpha^{1,\,0}}$ is the only representative with the above properties of its class in ${\cal E}^1_r(X)$. So, the map is well defined.

The map is injective because if $[\widehat{\alpha^{1,\,0}}]_{\bar\partial} = 0\in H^{1,\,0}_{\bar\partial}(X,\,\C)$, the $(1,\,0)$-form $\widehat{\alpha^{1,\,0}}$ must be $\bar\partial$-exact, hence zero. \hfill $\Box$

\vspace{2ex}

We can now give the generalisation to $E_r$-sGG manifolds of Proposition \ref{Prop:E_r-Albanese-map_hol}.

\begin{Prop}\label{Prop:E_r-Albanese-map_hol_E_r_sGG} Let $r\in\N^\star$ and let $X$ be a compact complex {\bf $E_r$-sGG manifold} with $\mbox{dim}_\C X =n$. Fix an arbitrary $\C$-basis of ${\cal E}^1_r(X)$: \begin{eqnarray*}\bigg\{\{\widehat{\alpha^{1,\,0}_1}\}_{dR}, \dots , \{\widehat{\alpha^{1,\,0}_{g_r}}\}_{dR}\bigg\},\end{eqnarray*} where $\widehat{\alpha^{1,\,0}_1},\dots , \widehat{\alpha^{1,\,0}_{g_r}}$ are the unique $d$-closed $C^\infty$ representatives of {\bf bidegree $(1,\,0)$} of these classes.

 Then, under the canonical identification (\ref{eqn:cal-E_r-1_can-isom}), the {\bf $E_r$-Albanese map} of $X$ is given by:  \begin{eqnarray}\label{eqn:E_r-Albanese-map_coordinates_E_r_sGG}\nonumber A_r:X &\longrightarrow & Alb_r(X)=\bigg({\cal E}^1_r(X)\bigg)^\star\Big/\mbox{Im}^{(r)}H_1(X,\,\Z)\simeq\C^{g_r}\Big/\Lambda, \\
    x & \longmapsto & \bigg(\int_{\xi_x}\widehat{\alpha^{1,\,0}_1},\dots , \int_{\xi_x}\widehat{\alpha^{1,\,0}_{g_r}}\bigg)   \hspace{3ex} \mbox{mod}\hspace{1ex} \Lambda,\end{eqnarray} where the lattice $\Lambda\subset\C^{g_r}$ is the group of periods of the $1$-forms $\widehat{\alpha^{1,\,0}_1},\dots , \widehat{\alpha^{1,\,0}_{g_r}}$, namely \begin{eqnarray*}\Lambda:=\bigg\{\bigg(\int_\gamma\widehat{\alpha^{1,\,0}_1},\dots , \int_\gamma\widehat{\alpha^{1,\,0}_{g_r}}\bigg)\,\bigg|\,[\gamma]\in\pi_1(X,\,a)\bigg\},\end{eqnarray*} and $\xi_x$ is any path in $X$ from the base point $a$ to $x$.

  In particular, the {\bf $E_r$-Albanese map} $A_r:X\longrightarrow Alb_r(X)$ is {\bf holomorphic}. 

\end{Prop}  

\noindent {\it Proof.} The proof is the same as that of Proposition \ref{Prop:E_r-Albanese-map_hol} after we replace the forms $\alpha_1^{1,\,0},\dots , \alpha_{g_r}^{1,\,0}$ by the forms $\widehat{\alpha^{1,\,0}_1},\dots , \widehat{\alpha^{1,\,0}_{g_r}}$. Indeed, since each $\widehat{\alpha^{1,\,0}_j}$ is smooth, $d$-closed and of bidegree $(1,\,0)$, it is locally of the form $\widehat{\alpha^{1,\,0}_j} = df_j$ for some locally defined {\it holomorphic} function $f_j$. \hfill $\Box$

\vspace{2ex}

We can also give the generalisation to $E_r$-sGG manifolds of Corollary \ref{Cor:A_r_surjective-submersion_c-par-page-ddbar}.

\begin{Cor}\label{Cor:A_r_surjective-submersion_c-par-E_r-sGG} If $X$ is a compact {\bf complex parallelisable} {\bf $E_r$-sGG manifold}, its $E_r$-Albanese map $A_r:X\longrightarrow Alb_r(X)$ is a holomorphic {\bf surjective submersion}.

\end{Cor}

\noindent {\it Proof.} The proof is the same as that of Corollary \ref{Cor:A_r_surjective-submersion_c-par-page-ddbar} after we replace a $\C$-basis $\{\alpha_1^{1,\,0},\dots , \alpha_{g_r}^{1,\,0}\}$ of $E_r^{1,\,0}(X)$ by a $\C$-basis $\{\widehat{\alpha^{1,\,0}_1},\dots , \widehat{\alpha^{1,\,0}_{g_r}}\}$ of ${\cal E}^1_r(X)$. (We identify each de Rham class in ${\cal E}^1_r(X)$ with its unique $d$-closed $C^\infty$ $(1,\,0)$-representative -- see Corollary \ref{Cor:cal-E_r-1_1-0_rep}.) The forms $\widehat{\alpha^{1,\,0}_1},\dots , \widehat{\alpha^{1,\,0}_{g_r}}$ remain linearly independent in $H_{\bar\partial}^{1,\,0}(X,\,\C)$ thanks to (ii) of Corollary \ref{Cor:cal-E-1-r_d-closed_1-0_canonical}.  \hfill $\Box$

\vspace{2ex}

It follows that the {\it page-(r-1)-$\partial\bar\partial$} hypothesis of Corollary \ref{Cor:A_r_c-par-page-ddbar_alg-dim_ineq} can be weakened to the {\it $E_r$-sGG} hypothesis.

\vspace{2ex}

We end this section with the following

\begin{Obs}\label{Obs:page_r-1_E_r-1_same-Albanese} Let $r\geq 2$ and suppose that a compact complex manifold $X$ is both {\bf page-$(r-1)$-$\partial\bar\partial$} and {\bf $E_{r-1}$-sGG}.

  Then, its $E_r$-Albanese torus $Alb_r(X)$ induced by the  page-$(r-1)$-$\partial\bar\partial$-property coincides with its  $E_{r-1}$-Albanese torus $Alb_{r-1}(X)$ induced by the $E_{r-1}$-sGG property.

\end{Obs}

\noindent {\it Proof.} On every compact complex manifold $X$ and for every integer $r\geq 2$, we have the inclusion: \begin{eqnarray}\label{eqn:E_r_R_r-1_inclusion}E_r^{0,\,1}(X)\subset E_{r-1}^{0,\,1}(X)\end{eqnarray} because the map $d_{r-1}^{-r+1,\,r-1}:E_{r-1}^{-r+1,\,r-1}(X)\longrightarrow E_{r-1}^{0,\,1}(X)$ is necessarily identically zero. Indeed, $r\geq 2$ implies $-r+1<0$, which gives $E_{r-1}^{-r+1,\,r-1}(X) = \{0\}$.

It follows that if $X$ is both {\it page-$(r-1)$-$\partial\bar\partial$} and {\it $E_{r-1}$-sGG}, we have canonical isomorphisms: \begin{eqnarray*}E_r^{1,\,0}(X)\oplus E_r^{0,\,1}(X)\simeq H^1_{dR}(X,\,\C)\simeq\overline{E_{r-1}^{0,\,1}(X)}\oplus E_{r-1}^{0,\,1}(X)\end{eqnarray*} and $E_r^{1,\,0}(X)\simeq\overline{E_{r-1}^{0,\,1}(X)}$. The contention follows. \hfill $\Box$

\section{Examples}\label{section:examples} We start by pointing out that the $E_r$-Albanese torus and map we constructed through cohomological methods for {\it page-(r-1)-$\partial\bar\partial$} and {\it $E_r$-sGG} manifolds in $\S$\ref{section:page_1}, respectively $\S$\ref{section:E_r-sGG}, are Hodge-theoretical realisations of the classical Albanese torus and map associated by Blanchard [Bla58] with any compact complex manifold $X$.

Indeed, according to [Uen75, chapter IV, $\S.9$], if $\{\gamma_1,\dots , \gamma_{b_1}\}$ is a basis of the free part of $H_1(X,\,\Z)$ and $\{\omega_1,\dots , \omega_q\}$ is a basis of the $\C$-vector space $H^{1,\,0}_d(X,\,\C)$ consisting of all {\bf $d$-closed} holomorphic $(1,\,0)$-forms on $X$, Blanchard considered the points: \begin{eqnarray*}c_1:=\bigg(\int\limits_{\gamma_1}\omega_1,\dots , \int\limits_{\gamma_1}\omega_q\bigg), \dots , \hspace{1ex} c_{b_1}:=\bigg(\int\limits_{\gamma_{b_1}}\omega_1,\dots , \int\limits_{\gamma_{b_1}}\omega_q\bigg)\in\C^q\end{eqnarray*} in $\C^q$ and the subgroup $\Delta:=\Z\, c_1 + \dots + \Z\, c_{b_1}\subset\C^q$ generated by them as a $\Z$-module. He then showed that the $\C$-vector subspace $\C\, c_1 + \dots + \C\, c_{b_1}\subset\C^q$ generated by $\Delta$ is the whole of $\C^q$ and considered the complex torus $\C^q/\overline{\Delta}$ that he then showed to be indeed an Albanese torus of $X$ (hence the Albanese torus $\mbox{Alb}(X)$, since it is unique up to analytic isomorphism). The notation $\overline{\Delta}$ stands for the smallest {\bf closed} Lie subgroup of $\C^q$ containing $\Delta$ such that the connected component of $\overline{\Delta}$ containing $0$ is a vector subspace of $\C^q$.  

  We showed, among other things, that:

\vspace{1ex}

(a)\, if $X$ is a {\it page-(r-1)-$\partial\bar\partial$-manifold}, then $E_r^{1,\,0}(X)\simeq H^{1,\,0}_d(X,\,\C)$ (cf. (ii) of Lemma \ref{Lem:E_r_1-0_singleton_ker});

\vspace{1ex}

(b)\, if $X$ is an {\it $E_r$-sGG manifold}, then ${\cal E}^1_r(X)\simeq H^{1,\,0}_d(X,\,\C)$ (cf. (i) of Corollary \ref{Cor:cal-E-1-r_d-closed_1-0_canonical}).

\vspace{1ex}

\noindent Moreover, in both cases, our Hodge-theoretical discussion of $\S$\ref{section:page_1}, respectively $\S$\ref{section:E_r-sGG}, implies that the image $\mbox{Im}^{(r)}H_1(X,\,\Z)$ of $H_1(X,\,\Z)$ in the dual space $E_r^{1,\,0}(X)^\star$, respectively ${\cal E}^1_r(X)^\star$, is a {\it lattice}. Thus, it identifies canonically with Blanchard's $\overline{\Delta}$.

\vspace{2ex}

Thanks to the canonical identifications of our $E_r$-Albanese torus and map with Blanchard's classical Albanese torus and map, we can use Proposition 3.16.1 on p. 53 of Winkelmann's memoir [Win98] to strengthen the conclusion of our Corollary \ref{Cor:A_r_c-par-page-ddbar_alg-dim_ineq} and its {\it $E_r$-sGG} analogue.

\begin{Cor}\label{Cor:A_r_E_r-sGG_alg-dim_eq} Let $X$ be a compact {\bf complex parallelisable $E_r$-sGG manifold} and let $A_r:X\longrightarrow Alb_r(X)$ be its $E_r$-Albanese map. Then, we have the following equality of algebraic dimensions: \begin{eqnarray}\label{eqn:A_r_E_r-sGG_alg-dim_eq}a\bigg(Alb_r(X)\bigg)= a(X).\end{eqnarray}

\end{Cor}  

\noindent {\it Proof.} The inequality ``$\leq$'' of (\ref{eqn:A_r_E_r-sGG_alg-dim_eq}) was proved in Corollary \ref{Cor:A_r_c-par-page-ddbar_alg-dim_ineq} and, as previously explained, still holds under the mere $E_r$-sGG hypothesis in place of the page-(r-1)-$\partial\bar\partial$ assumption.

On the other hand, the fact that the compact manifold $X$ is complex parallelisable is equivalent to the existence of a connected, simply connected, {\it complex} Lie group $G$ and of a lattice $\Gamma\subset G$ such that $X=G/\Gamma$.

For such an $X=G/\Gamma$, Proposition 3.16.1 of [Win98] ensures that:

\vspace{1ex}

(a)\, the Albanese torus $Alb(X)$ coincides with the quotient $G/I$, where $I$ is the smallest {\it closed} complex Lie subgroup of $G$ containing $G'\,\Gamma$ and $G'$ is the commutator subgroup of $G$ -- this follows from a universality property of $G/I$ (cf. (1) of [Win98, Proposition 3.16.1]);

\vspace{1ex}

(b)\, if $A:X=G/\Gamma\longrightarrow Alb(X) = G/I$ is the Albanese map of $X$, then every meromorphic function $f:X\dashrightarrow\C$ on $X$ is the pullback $A^\star g = g\circ A$ of a meromorphic function $g:Alb(X)\dashrightarrow\C$ on its Albanese torus -- see (4) of [Win98, Proposition 3.16.1].

\vspace{1ex}

This property (b) implies at once that $a(X)\leq a(Alb(X))$, proving the inequality ``$\geq$'' in (\ref{eqn:A_r_E_r-sGG_alg-dim_eq}). (Recall that $Alb_r(X)$ is canonically isomorphic to the classical $Alb(X)$.)  \hfill $\Box$

\subsection{Non-trivial examples of $E_r$-Albanese tori and maps}\label{subsection:non-trivial_examp}

\subsubsection{The Iwasawa manifold $I^{(3)}$}\label{subsubsection: Iwasawa} We briefly discussed it in Corollary \ref{Cor:Iwasawa_alg-dim_2}. Let us now further point out that, besides being a {\it page-$1$-$\partial\bar\partial$}-manifold (which leads to the $E_2$-Albanese torus $Alb_2(I^{(3)})$), $I^{(3)}$ is also an {\it sGG-manifold} ([PU, Corollary 1.5]) -- which leads to the $E_1$-Albanese torus $Alb_1(I^{(3)})$. Thanks to Observation \ref{Obs:page_r-1_E_r-1_same-Albanese}, these tori coincide: \begin{eqnarray*}Alb_1(I^{(3)}) =  Alb_2(I^{(3)}):=Alb(I^{(3)}).\end{eqnarray*} Moreover, $\mbox{dim}_\C Alb(I^{(3)}) = 2$. 

It is standard that $Alb(I^{(3)})$ is the complex $2$-dimensional projective torus $\C/\Z[i] \times \C/\Z[i]$, while the fibre of the Albanese map $A: = A_1 = A_2:I^{(3)}\longrightarrow Alb(I^{(3)})$ is the Gauss elliptic curve $\C/\Z[i]$.

\subsection{Case where $b_1(X)=0$}\label{subsection:examp_b_1=0} The complex dimension of the $E_r$-Jacobian and $E_r$-Albanese tori of any {\it page-(r-1)-$\partial\bar\partial$-manifold} and, more generally, of any {\it $E_r$-sGG manifold}, is $g_r= (1/2)\,b_1(X)$.

In particular, when $b_1(X)=0$, these tori and the associated $E_r$-Albanese map are trivial. More generally, the classical Albanese torus and map of every compact complex manifold $X$ with $b_1(X)=0$ (in particular, whenever $X$ is simply connected) are trivial. It then seems natural to look for alternative, but similar, objects constructed from the Dolbeault cohomology group $H^{0,\,1}_{\bar\partial}(X,\,\C)$ when it is non-trivial (a thing that can only happen when $X$ is not a $\partial\bar\partial$-manifold, if $b_1(X)=0$).

\vspace{1ex}

Recall that for any compact complex manifold $X$, the {\it short exponential exact sequence}: \begin{eqnarray*}0\longrightarrow \underline{\Z} \longrightarrow {\cal O}_X & \longrightarrow & {\cal O}_X^\star \longrightarrow 0 \\
  f & \longmapsto & e^{2\pi i f} \end{eqnarray*} induces a long exact sequence in cohomology: \begin{eqnarray*}\dots\longrightarrow H^1(X,\,\Z)\longrightarrow H^1(X,\,{\cal O}_X)\simeq H^{0,\,1}_{\bar\partial}(X,\,\C)\longrightarrow  H^1(X,\,{\cal O}_X^\star) = \mbox{Pic}\,(X) \xrightarrow{c_1} H^2(X,\,\Z)\longrightarrow\dots,\end{eqnarray*} where $\mbox{Pic}\,(X)=H^1(X,\,{\cal O}_X^\star)$ is the {\it Picard group} of $X$ (consisting of the isomorphism classes of holomorphic line bundles on $X$) and the map $c_1$ associates with each such class its {\it first Chern class}.

This yields the following canonical isomorphism for $\mbox{Pic}^0\,(X):=\ker c_1\subset\mbox{Pic}\,(X)$: \begin{eqnarray*}\mbox{Pic}^0\,(X)\simeq H^{0,\,1}_{\bar\partial}(X,\,\C)\Big/\mbox{Im}\,\bigg(H^1(X,\,\Z)\bigg).\end{eqnarray*} This shows, in particular, that $\mbox{Pic}^0\,(X)$ is a compact complex torus if and only if $b_1(X) = 2\,h^{0,\,1}_{\bar\partial}(X)$. The last condition was proved to be equivalent to $X$ being an {\it sGG manifold} (equivalently, an {\it $E_1$-sGG manifold}) in [PU, Theorem 1.6]. In particular, $\mbox{Pic}^0\,(X)$ is the {\it Jacobian torus} of $X$ (which coincides with the {\it $E_1$-Jacobian torus} of $X$) when $X$ is an $sGG$ manifold (cf. [Pop18] and the above Definition \ref{Def:E_r-Jacobi-Albanese-tori_E_r-sGG} in the case $r=1$).

In particular, if we assume that $H^2(X,\,\Z) = \{0\}$, every holomorphic line bundle $L$ on $X$ has zero first Chern class (i.e. $\mbox{Pic}\,(X)= \mbox{Pic}^0\,(X)$). If we also assume that  $H^1(X,\,\Z) = \{0\}$, the second relation below is a {\it canonical isomorphism}: \begin{eqnarray}\label{eqn:Pic-Pic0_b1_b2_0}\mbox{Pic}\,(X)= \mbox{Pic}^0\,(X)\simeq H^{0,\,1}_{\bar\partial}(X,\,\C)\simeq\C^g,\end{eqnarray} where $g:=h^{0,\,1}_{\bar\partial}(X)$. Thus, the Picard group $\mbox{Pic}\,(X)$ is a $\C$-vector space in this case. In particular, every holomorphic line bundle $L$ on $X$ admits a {\it unique square root} $L^{1/2}$ (denoted by $(1/2)\, L$ in the additive notation of $\mbox{Pic}\,(X)\simeq\C^g$) because the multiplication by any non-zero scalar is an automorphism of a vector space. This implies that $X$ admits a {\it unique spin structure}. (Recall that on a compact complex manifold $X$, {\it spin structures} identify with square roots of the canonical line bundle $K_X$.)

\subsubsection{The Calabi-Eckmann manifolds $X_{p,\,q}$}\label{subsubsection:C-E_S6} For all integers $p,q\geq 1$, let $X_{p,\,q}$ be the compact complex manifold consisting of the $C^\infty$ manifold $S^{2p+1}\times S^{2q+1}$ (a product of two odd-dimensional spheres) endowed with the {\bf Calabi-Eckmann} complex structure. It was shown in [PSU20b, Proposition 5.3 and Lemma 5.4] that each $X_{p,\,q}$ is an {\it $E_2$-sGG manifold} (because $E_2^{n,\,n-1}(X_{p,\,q})= \{0\}$, where $n=\mbox{dim}_\C X_{p,\,q} = p+q+1$), but not a page-$(r-1)$-$\partial\bar\partial$-manifold for any $r\geq 1$ (because the de Rham cohomology of $X_{p,\,q}$ is not pure -- see Theorem \ref{The:page_r-1_pure-characterisation}). The manifolds $X_{p,\,q}$ are not sGG ($=E_1$-sGG) either, for cohomological dimensional reasons: $b_1(X_{p,\,q}) = 0$ but $h^{0,\,1}_{\bar\partial}(X_{p,\,q}) =1$.

Thus, the $E_2$-Albanese torus and map of each $X_{p,\,q}$ are well defined (thanks to $\S$\ref{section:E_r-sGG}), but trivial. Since $b_1(X_{p,\,q}) = b_2(X_{p,\,q}) = 0$, the Calabi-Eckmann manifolds $X_{p,\,q}$ have properties (\ref{eqn:Pic-Pic0_b1_b2_0}) with $g=1$.

\vspace{1ex}

On the other hand, by [CE53, Theorem II], each manifold $X_{p,\,q}$ admits a complex analytic fibration $\sigma : X_{p, \, q}\to\Proj^p\times\Proj^q$  over a product of two complex projective spaces. The fibres are complex $1$-dimensional tori (i.e. elliptic curves). The holomorphicity and the surjectivity of $\sigma$ imply, by the same standard argument we used in the proof of Corollary \ref{Cor:A_r_c-par-page-ddbar_alg-dim_ineq}, that $a(\Proj^p\times\Proj^q)\leq a(X_{p, \, q})$. Meanwhile, $\Proj^p\times\Proj^q$ is projective, hence Moishezon, hence its algebraic dimension is maximal: $a(\Proj^p\times\Proj^q) = p+q$, leaving only two possibilities for the algebraic dimension of $X_{p, \, q}$, $p+q$ or $p+q+1 = \mbox{dim}_\C X_{p, \, q}$. However, the latter case never occurs because, otherwise, $X_{p, \, q}$ would be Moishezon, hence a {\it $\partial\bar\partial$-manifold}. But this is known not to hold (see the above reference to [PSU20b]). For example, $X_{p, \, q}$ is not even a strongly Gauduchon ($=$ sG $=E_1$-sG) manifold -- see [Pop14, $\S2$, (a)].

We conclude that $a(X_{p, \, q}) =p+q$.

\subsubsection{The $6$-dimensional sphere $S^6$}\label{subsubsection:S6} 
Let us denote by $X$ the $6$-dimensional sphere $S^6$ endowed with any hypothetical complex structure. It was shown in [Gra97, Lemma 4] that the Dolbeault cohomology group $H^{3,0}(X)$ is zero, so by Serre duality 
\begin{eqnarray}\label{eqn:H30-S6}
h^{3,0}(X)=h^{0,3}(X)=0.
\end{eqnarray}
These equalities imply, by [Uga00, Lemma 2.1], that the Fr\"olicher spectral sequence of $X$ degenerates at $E_3$. (This is denoted by $E_3(X)= E_{\infty}(X)$.)

On the other hand, in [Gra97, Theorem 5] it is proved that $h^{0,1}(X)\geq 1$. Furthermore, by [Uga00, Proposition 3.1] the following equalities for the Hodge numbers of $X$ hold:
\begin{eqnarray}\label{eqn:H01-S6}
h^{0,1}(X) = h^{0,2}(X) + 1,\quad\quad 
h^{1,1}(X) + h^{2,0}(X) = h^{1,0}(X) + h^{1,2}(X) + 1.
\end{eqnarray}
Therefore, $E_1(X)\not= E_{\infty}(X)$, so $X$ is not a (page-$0$-)$\partial\bar\partial$-manifold. 
In addition, since the first Betti number of $S^6$ is zero, we have $b_1(X)\ne 2\, h^{0,1}(X)$, which implies that $X$ is not ($E_1$-)sGG.

\begin{Prop}\label{Prop:S6-precise-r-2-or-3} 
Let $X$ be the $6$-dimensional sphere $S^6$ endowed with any hypothetical complex structure. Then, $X$ is always a page-$2$-$\partial\bar\partial$-manifold, hence also an $E_3$-sGG manifold. Furthermore:
\\[4pt]
\vspace{1ex}
\ \ (i)\, 
$X$ is $E_2$-sGG if and only if $h^{1,1}(X) = h^{1,2}(X) + 1$. 
\\
\vspace{1ex}
(ii)\, 
$X$ is a page-$1$-$\partial\bar\partial$-manifold if and only if $h^{1,1}(X) = h^{1,2}(X) + 1$ and $e_2^{0,\,2}(X)=0$.
\end{Prop}
\noindent {\it Proof.} 
As mentioned above, the Fr\"olicher spectral sequence of $X$ satisfies $E_3(X)= E_{\infty}(X)$. Since $H^k_{dR}(X,\,\C)=\{0\}$ for every $1\leq k\leq 5$, the de Rham cohomology of $X$ is {\it pure} (see (\ref{eqn:de-Rham_pure_def})). It follows from Theorem \ref{The:page_r-1_pure-characterisation} that $X$ is a page-$2$-$\partial\bar\partial$-manifold. Hence, by Proposition~\ref{Prop:page_r-1_ddbar_E_r-sGG}, $X$ is also $E_3$-sGG.

For the rest of the proof, we will use the properties obtained in [Uga00, Theorem 3.2] for the dimensions $e_2^{p,\,q}(X)=\mbox{dim}_\C E_2^{p,\,q}(X)$, $0\leq p,\, q\leq 3$, of the spaces on the second page of the Fr\"olicher spectral sequence of $X$. Concretely, we have:
\begin{equation}\label{eqn:e2-S6}
\begin{array}{l}
e_2^{0,\,0}(X)=e_2^{3,\,3}(X)=1,\\[5pt]
e_2^{1,\,0}(X)=e_2^{1,\,1}(X)=e_2^{3,\,0}(X)=e_2^{0,\,3}(X)=e_2^{2,\,2}(X)=e_2^{2,\,3}(X)=0,\\[5pt]
p(X) \geq e_2^{0,\,1}(X)=e_2^{2,\,0}(X)=e_2^{1,\,3}(X)=e_2^{3,\,2}(X) \geq 0,\\[5pt]
q(X) \geq e_2^{0,\,2}(X)=e_2^{2,\,1}(X)=e_2^{1,\,2}(X)=e_2^{3,\,1}(X) \geq 0,\\[5pt]
e_2^{0,\,1}(X) = h^{1,2}(X) - h^{1,1}(X) + 1 \geq 0,
\end{array}
\end{equation}
where $p(X)=\min\{h^{0,\,1}(X),h^{2,\,0}(X)\}$, 
and $q(X)=\min\{h^{0,\,2}(X),h^{1,\,2}(X)\}$.

For the proof of {\it (i)} we use Theorem~\ref{The:E_r-sGG_Hodge-decomp}. The property $E_2$-sGG is equivalent to $b_1(X)=0=2\,e_2^{0,\,1}(X)$, so from the last equality in \eqref{eqn:e2-S6} it is equivalent to $h^{1,2}(X) - h^{1,1}(X) + 1=e_2^{0,\,1}(X)=0$.

Finally, by Theorem \ref{The:page_r-1_pure-characterisation} and the purity of the de Rham cohomology of $X$, the proof of {\it (ii)} is reduced to finding the conditions under which $E_2(X)= E_{\infty}(X)$. This, in turn, is equivalent to $e_2^{0,\,0}(X)= e_2^{3,\,3}(X)=1$ and $e_2^{p,\,q}(X)=0$ for the other values of $(p,q)$. So, the result follows from \eqref{eqn:e2-S6}.
\hfill $\Box$

\vspace{2ex}

As for the metric properties of $X$, the following result gives a restriction for the existence of SKT metrics. Recall that a Hermitian metric $\omega$ is said to be SKT if $\partial\bar\partial\omega=0$. Meanwhile, a Hermitian metric $\omega$ on an $n$-dimensional complex manifold is said to be {\it astheno-K\"ahler} (cf. [JY93]) if $\partial\bar\partial\omega^{n-2}=0$. When $n=3$, as is the case with our $X$, these two notions coincide.

\begin{Prop}\label{Prop:S6-SKT-metrics} Let $X$ be $S^6$ endowed with any hypothetical complex structure $J$. If $X$ supports an SKT metric, then $X$ is a {\bf page-$1$-$\partial\bar\partial$-manifold}. 
\end{Prop}
\noindent {\it Proof.} 
We will use a necessary condition, proved in [CR23, Theorem A], for the existence of astheno-K\"ahler metrics in terms of the Bott-Chern and Aeppli numbers of any compact complex manifold. In the case of $X=(S^6,\,J)$, this translates to: if $X$ admits an SKT metric, then
$$
h_{BC}^{0,1}(X) \leq h_{A}^{0,1}(X) \leq h_{BC}^{0,1}(X) + 1.
$$

On the other hand, the following properties of these numbers are proved in [McH17, Section 3] when $X$ is the 6-sphere endowed with any hypothetical complex structure: 
$$h_{BC}^{0,1}(X)=0,\quad\quad h_{A}^{0,1}(X)=h^{2,0}(X)+h^{0,2}(X)+1.$$

Now, suppose that $X=(S^6,\,J)$ is SKT.  Jost and Yau proved in [JY93] that every holomorphic 1-form on (any astheno-K\"ahler) $X$ is closed. This means that $h^{1,0}(X)=h_{BC}^{1,0}(X)$. Therefore, 
$$h^{1,0}(X)=h_{BC}^{1,0}(X)=h_{BC}^{0,1}(X)=0,$$ 
which, together with the condition,  
$$
h^{2,0}(X)+h^{0,2}(X)+1=h_{A}^{0,1}(X) \leq h_{BC}^{0,1}(X) + 1=1,
$$
implies that $h^{2,0}(X)+h^{0,2}(X) \leq 0$, i.e. $h^{2,0}(X)=0$ and $h^{0,2}(X)=0$. 
By \eqref{eqn:e2-S6}, this forces $p(X)=0=q(X)$, so $e_2^{0,\,2}(X)=0$ and $e_2^{0,\,1}(X)=0=h^{1,2}(X)-h^{1,1}(X)+1$. By {\it (ii)} in Proposition~\ref{Prop:S6-precise-r-2-or-3}, $X$ is a page-$1$-$\partial\bar\partial$-manifold.
\hfill $\Box$

\vspace{2ex}

The above results are valid for {\it any} hypothetical complex structure on $S^6$. In the following remarks, we observe some properties which depend on the value of the Hodge number $h^{1,1}(X)$.

\begin{Rem}\label{Rem:S6-h11-ne-0} It is announced in [A26] (see [Eng26] for a description of the geometric ideas behind the construction and a self-contained proof) that a complex structure exists on $S^6$ with the following Hodge numbers: \begin{eqnarray*}h^{0,1}(Y)=1\hspace{5ex} \mbox{and} \hspace{5ex} h^{1,0}(Y)=h^{2,0}(Y)=h^{0,2}(Y)=0.\end{eqnarray*}

We denote by $Y$ the compact complex manifolds of [A26] and reserve $X$ for a generic (hypothetical) complex structure on the 6-sphere. 
Then, by \eqref{eqn:e2-S6}, one has: \begin{eqnarray*} e_2^{0,\,1}(Y)=0 \hspace{5ex} \mbox{and} \hspace{5ex} e_2^{0,\,2}(Y)=0,\end{eqnarray*} so $h^{1,1}(Y)=h^{1,2}(Y)+1$. It follows that $Y$ is a {\bf page-$1$-$\partial\bar\partial$-manifold}, by (ii) in Proposition~\ref{Prop:S6-precise-r-2-or-3}.

\end{Rem}

\vspace{1ex}

As for the metric properties of the complex manifolds $Y$, we get the following

\begin{The}\label{The:non-sG_S6} Let $Y$ be a compact complex manifold claimed in [A26] to be homeomorphic to $S^6$.

  Then, $Y$ carries {\bf no strongly Gauduchon} (hence {\bf no balanced}) metric.  

\end{The}

\noindent {\it Proof.} By [A26] and [Eng26], there is a surjective holomorphic map $f:Y\longrightarrow \Proj^1$. We consider the real $(1,1)$-form $\alpha=f^*(\omega_{\mathrm{FS}})$ on $Y$, where $\omega_{\mathrm{FS}}$ is the Fubini-Study metric of $\Proj^1$. Then, $d\alpha=0$ and $\alpha\geq 0$.

Thus, for any Hermitian metric $\omega$ on $Y$, we get: \begin{eqnarray}\label{eqn:integral_alpha-omega-square_positivie}0 < \int_Y \alpha\wedge\omega^2.\end{eqnarray} Indeed, $\alpha\wedge\omega^2 = 2\,(\Lambda_\omega\alpha)\,dV_\omega\geq 0$ everywhere on $Y$ since $\Lambda_\omega\alpha = \lambda_1 + \lambda_2 + \lambda_3$ is the sum of the eigenvalues of $\alpha$ with respect to $\omega$ and all the eigenvalues $\lambda_j$ are non-negative everywhere due to $\alpha$ being semi-positive. This implies that we have at least $\int_Y \alpha\wedge\omega^2\geq 0$ and equality occurs if and only if $\lambda_j$ vanishes identically on $Y$ for every $j=1,2,3$. However, this last possibility cannot occur since, otherwise, $\alpha$ would be the zero form on $Y$. This is impossible because $\omega_{FS}>0$ everywhere on $\Proj^1$ and because $f$ is surjective, so $df$ has rank $1$ somewhere, so $f^\star\omega_{FS}$ is non-zero somewhere. 

\vspace{1ex}

On the other hand, $\alpha$ is $d$-closed, so it defines a de Rham cohomology class $\{\alpha\}_{dR}\in H^2_{dR}(Y,\,\C)$. Since $H^2_{dR}(Y,\,\C) = \{0\}$ and $\alpha$ is real, there exists a $C^\infty$ real $1$-form $\beta = \beta^{1,\,0} + \beta^{0,\,1}$ on $Y$ such that $\alpha=d\beta$. This means that $\beta^{1,\,0} = \overline{\beta^{0,\,1}}$ and \begin{eqnarray*}\partial\beta^{1,\,0} = 0, \hspace{5ex} \partial\beta^{0,\,1} + \bar\partial\beta^{1,\,0} = \alpha, \hspace{5ex} \bar\partial\beta^{0,\,1} = 0.\end{eqnarray*}

In particular, we get: \begin{eqnarray*}\int_Y \alpha\wedge\omega^2 = \int_Y\partial\beta^{0,\,1}\wedge\omega^2 + \int_Y\bar\partial\beta^{1,\,0}\wedge\omega^2:=I_1 + I_2,\end{eqnarray*} with $I_2=\overline{I_1}$. We will show that $I_1=0$ if $\omega$ were {\it strongly Gauduchon}. This would imply that $\int_Y \alpha\wedge\omega^2 = 0$ in that case, contradicting (\ref{eqn:integral_alpha-omega-square_positivie}).

\vspace{1ex}

Suppose that $\omega$ is {\it strongly Gauduchon}. Then, there exists $\eta\in C^\infty_{3,\,1}(Y,\,\C)$ such that $\partial\omega^2 = \bar\partial\eta$ on $Y$. (Recall that $\mbox{dim}_\C Y = 3$.) Two applications of the Stokes theorem yield: \begin{eqnarray*}I_1 = \int_Y\partial\beta^{0,\,1}\wedge\omega^2 = \int_Y\beta^{0,\,1}\wedge\partial\omega^2 = \int_Y\beta^{0,\,1}\wedge\bar\partial\eta = \int_Y\bar\partial\beta^{0,\,1}\wedge\eta = 0,\end{eqnarray*} where the last equality follows from $\bar\partial\beta^{0,\,1} = 0$.   \hfill $\Box$

\vspace{2ex}

 However, note that there always exists an {\bf $E_2$-sG} metric on $Y$, because $Y$ is page-$1$-$\partial\bar\partial$ (see Proposition~\ref{Prop:page_r-1_ddbar_E_r-sGG}).

\begin{Rem}\label{Rem:S6-h11-0}
  By Remark~\ref{Rem:S6-h11-ne-0}, the Hodge number $h^{1,1}(Y)$ is always non-zero for the complex manifolds $Y$ of [A26].

  Now, suppose that a complex structure $J$ with $h^{1,1}=0$ exists on $S^6$. By \eqref{eqn:e2-S6}, this would imply that $E_2\not= E_{\infty}$, because $e_2^{0,1}\ne 0$. By Proposition~\ref{Prop:S6-precise-r-2-or-3}, the complex manifold $X=(S^6,\,J)$ would be a page-$2$-$\partial\bar\partial$-manifold (hence also an $E_3$-sGG manifold), but not a page-$1$-$\partial\bar\partial$-manifold. Therefore, Proposition~\ref{Prop:S6-SKT-metrics} implies that the manifold $X=(S^6,\,J)$ would not admit any SKT metric.
\end{Rem}

\vspace{2ex}

The last remark ties in with Question \ref{introd_Question:S6_SKT} asked in the introduction.

\vspace{3ex}

\noindent {\bf References.} \\

\vspace{1ex}

\noindent 
[A26]\, L. Alp\"oge --- {\it The $(3, 4, \infty)$ modular family of 2-tori, completed
at its three special points, is a complex structure on $S^6$} --- available at https:
//alpo.ge/s6.pdf, 108 pp., 2026.



\vspace{1ex}

\noindent [Bla58]\, A. Blanchard --- {\it Sur les vari\'et\'es analytiques complexes} --- Ann. Sci. \'Ecole Norm. Sup. {\bf 73} (1958), 157-202.

\vspace{1ex}

\noindent [CE53]\, E. Calabi, B. Eckmann --- {\it A Class of Compact, Complex Manifolds Which Are Not Algebraic} --- Ann. of Math. {\bf 58} (1953) 494-500.

\vspace{1ex}

\noindent 
[CR23]\, I. Chiose, R. R\u{a}sdeaconu --- 
{\it Remarks on astheno-K\"aler manifolds, Bott-Chern and Aeppli cohomology groups} --- Ann. Global Anal. Geom. {\bf 63} (2023), no. 3, Paper No. 24, 23 pp.

\vspace{1ex}

\noindent [Dem97]\, J.-P. Demailly --- {\it Complex Analytic and Algebraic Geometry} ---

\noindent http://www-fourier.ujf-grenoble.fr/~demailly/books.html

\vspace{1ex}

\noindent 
[Eng26]\, P. Engel ---  {\it Complex structures on $S^6$} --- 
available at https://philip-engel.github.io/S6.pdf, 25 pp., 2026.

\vspace{1ex}

\noindent [Gau77] P. Gauduchon --- {\it Le th\'eor\`eme de l'excentricit\'e nulle} --- C. R. Acad. Sc. Paris, S\'erie A, t. {\bf 285} (1977), 387-390.

\vspace{1ex}

\noindent [Gra97]\, A. Gray -- {\it A property of a hypothetical complex structure on the six sphere} --- 
Boll. Un. Mat. Ital. B (7) {\bf 11}, suppl. fasc. 2 (1997), 251-255.

\vspace{1ex}

\noindent [JY93]\, 
J. Jost, S-T. Yau --- {\it A non-linear elliptic system for maps from Hermitian to Riemannian
manifolds and rigidity theorems in Hermitian geometry} --- 
Acta Math. {\bf 170} (1993), 221-54; Corrigendum Acta Math. {\bf 173} (1994), 307.

\vspace{1ex}

\noindent [KS23]\, H. Kasuya, J. Stelzig --- {\it Fr\"olicher Spectral Sequence and Hodge Structures on the Cohomology of Complex Parallelisable Manifolds.} --- Transform. Groups {\bf 28} (2023), no. 4, 1595-1606. 

\vspace{1ex}

\noindent 
[LR26]\, W. Liu, S. Rollenske --- {\it Yet Another Family of Complex Structures on $S^6$} ---arXiv:2609.33826 [math.DG]; https://doi.org/10.48550/arXiv.2609.33826

\vspace{1ex}

\noindent [McH17]\, A. McHugh --- {\it Narrowing cohomologies on complex $S^6$} 
--- Eur. J. Pure Appl. Math. {\bf 10} (2017), no. 3, 440-454.

\vspace{1ex}

\noindent [Pop13]\, D. Popovici --- {\it Deformation Limits of Projective Manifolds: Hodge Numbers and Strongly Gauduchon Metrics} --- Invent. Math. {\bf 194} (2013), 515-534.

\vspace{1ex}

\noindent [Pop14]\, D. Popovici --- {\it Deformation Openness and Closedness of Various Classes of Compact Complex Manifolds; Examples} --- Ann. Sc. Norm. Super. Pisa Cl. Sci. (5), Vol. {\bf XIII} (2014), 1-51. 

\vspace{1ex}

\noindent [Pop18]\, D. Popovici --- {\it The Albanese Map of sGG Manifolds and Self-Duality of the Iwasawa Manifold} --- Riv. Mat. Univ. Parma, Vol. 9 (2018), 177-190.

\vspace{1ex}

\noindent [Pop19]\, D. Popovici --- {\it Adiabatic Limit and Deformations of Complex Structures} --- arXiv e-print AG 1901.04087v5

\vspace{1ex}

\noindent [PU]\,  D. Popovici, L. Ugarte --- {\it Compact Complex Manifolds with Small Gauduchon Cone} --- Proc. London Math. Soc. (3) 116 (2018), no. 5, 1161-1186.

\vspace{1ex}

\noindent [PSU20a]\,  D. Popovici, J. Stelzig, L. Ugarte --- {\it Higher-Page Hodge Theory of Compact Complex Manifolds} --- arXiv e-print AG 2001.02313v3; Ann. Sc. Norm. Super. Pisa Cl. Sci. (5) Vol. XXV (2024), 1431-1464. 

\vspace{1ex}

\noindent [PSU20b]\, D. Popovici, J. Stelzig, L. Ugarte --- {\it Higher-Page Bott-Chern and Aeppli Cohomologies and Applications} --- J. reine angew. Math., doi: 10.1515/crelle-2021-0014.

\vspace{1ex}

\noindent [PSU20c]\, D. Popovici, J. Stelzig, L. Ugarte ---  {\it Deformations of Higher-Page Analogues of $\partial\bar\partial$-Manifolds} --- Math. Z. (2021), https://doi.org/10.1007/s00209-021-02861-0.

\vspace{1ex}

\noindent [Uen75]\, K. Ueno --- {\it Classification Theory of Algebraic Varieties and Compact Complex Spaces} --- LNM {\bf 439} (1975).  

\vspace{1ex}

\noindent [Uga00]\, L. Ugarte --- {\it Hodge numbers of a hypothetical complex structure on the six sphere} --- 
Geom. Dedicata {\bf 81} (2000), no. 1-3, 173-179.

\vspace{1ex}

\noindent 
[Via26]\,  J. Viaclovsky --- {\it A two-parameter family of complex structures on $S^6$} --- arXiv:2609.33785 [math.AG]; https://doi.org/10.48550/arXiv.2609.33785

\vspace{1ex}

\noindent [Win98]\, J. Winkelmann --- {\it Complex Analytic Geometry of Complex Parallelizable Manifolds} ---  M\'em. Soc. Math. Fr. (N.S.) No. {\bf 72-73} (1998).

\vspace{3ex}

\noindent Institut de Math\'ematiques de Toulouse, Universit\'e Paul Sabatier,

\noindent 118 route de Narbonne, 31062 Toulouse, France

\noindent Email: popovici@math.univ-toulouse.fr

\vspace{2ex}

\noindent and

\vspace{2ex}

\noindent Departamento de Matem\'aticas\,-\,I.U.M.A., Universidad de Zaragoza,

\noindent Campus Plaza San Francisco, 50009 Zaragoza, Spain

\noindent Email: ugarte@unizar.es

\end{document}